\documentclass[a4paper]{article}
\usepackage{amsmath,amssymb,amsthm}
\usepackage{mathdots}
\usepackage[margin=1in]{geometry}
\usepackage[makeroom]{cancel}
\usepackage{bbm}            
\usepackage{graphicx,color}
\usepackage{hyperref,url}   

\renewcommand{\leq}{\leqslant}
\renewcommand{\geq}{\geqslant}
\newcommand{\colonequal}{\mathrel{\mathop:}=}

\newcommand{\N}{\mathbbm N}
\newcommand{\Z}{\mathbbm Z}
\newcommand{\Q}{\mathbbm Q}

\newtheorem{thm}{Theorem} 
\newtheorem{prop}[thm]{Proposition}
\newtheorem{lem}[thm]{Lemma} 
\newtheorem{cor}[thm]{Corollary}
\newtheorem{conj}[thm]{Conjecture}
\newtheorem{defn}[thm]{Definition}

\numberwithin{equation}{section} 

\parskip \smallskipamount

\begin{document}


\title{A Curious Family of Binomial Determinants That Count Rhombus Tilings of a Holey Hexagon}

\author{Christoph~Koutschan\thanks{Supported by the Austrian Science Fund (FWF): P29467-N32 and F5011-N15.}
        \\
        Johann Radon Institute for Computational and\\
        Applied Mathematics (RICAM), Austrian Academy of Sciences\\
        Altenberger Stra\ss e 69, A-4040 Linz, Austria
        \and
        Thotsaporn Thanatipanonda\\
        Science Division, Mahidol University International College,\\
        Nakhonpathom, Thailand 73170
        }
\date{}

\maketitle

\begin{abstract}
We evaluate a curious determinant, first mentioned by George Andrews in 1980
in the context of descending plane partitions. Our strategy is to combine the
famous Desnanot-Jacobi-Dodgson identity with automated proof techniques. More
precisely, we follow the holonomic ansatz that was proposed by Doron
Zeilberger in 2007. We derive a compact and nice formula for Andrews's
determinant, and use it to solve a challenge problem that we posed in a
previous paper. By noting that Andrews's determinant is a special case of a
two-parameter family of determinants, we find closed forms for several
one-parameter subfamilies. The interest in these determinants arises because
they count cyclically symmetric rhombus tilings of a hexagon with several
triangular holes inside.
\end{abstract}


\section{Introduction}\label{sec.intro}

Plane partitions were a hot topic back in the 1970's and 1980's (as
beautifully described in \cite{Bressoud99}), and they still keep
combinatorialists busy. For example, the $q$-enumeration formula of totally
symmetric plane partitions, conjectured independently by David Robbins and
George Andrews in 1983, remained open for almost 30 years and was finally
proved in 2011~\cite{KoutschanKauersZeilberger11} using massive computer algebra
calculations. The problem that we treat in this paper originates around the same
time, when combinatorialists started to employ determinants to reformulate the
counting problem of plane partitions.

The following determinant counts descending plane partitions, and it was
famously evaluated by George Andrews~\cite{Andrews79} in~1979:
\begin{equation}\label{eq.AndDet}
  \det_{1\leq i,j\leq n} \left(\delta_{i,j} + \binom{\mu+i+j-2}{j-1}\right),
\end{equation}
where $\delta_{i,j}$ denotes the Kronecker delta, i.e., $\delta_{i,j}=1$ if
$i=j$ and $\delta_{i,j}=0$ otherwise. The same determinant is also mentioned
in Krattenthaler's classic treatise on
determinants~\cite[Thm.~32]{Krattenthaler99} (where $\mu$ is replaced by
$2\mu$).  One year later, in 1980, Andrews~\cite[page 105]{Andrews80a} came up
with a curious determinant which is a slight variation of the above:
\[
  D(n) := \det_{1 \leq i,j \leq n}\left( \delta_{i,j}+{ \binom{\mu+i+j-2}{j}}\right).
\]
He conjectured a closed-form formula for the quotient $D(2n)/D(2n-1)$.
It was mentioned again (and popularized) as Problem 34 in Krattenthaler's
complement \cite{Krattenthaler05}, and it was proven, for the first
time, by the authors of the present paper in
2013~\cite{KoutschanThanatipanonda13}.

However this proves only ``half'' of the formula for~$D(n)$. The quotient
$D(2n+1)/D(2n)$ remained mysterious, due to an increasingly large ``ugly''
(i.e., irreducible) polynomial factor that is always shared between two
consecutive determinants. Thus the determinant~$D(n)$ does not completely
factor into linear polynomials, while many similar determinants do.  Not fully
satisfied with this situation, the first-named author made a
\textit{monstrous} conjecture~\cite[Conj.~6]{KoutschanThanatipanonda13} of the
full formula for~$D(n)$.  In this paper, we derive and prove a nicer formula
for~$D(n)$ (Section~\ref{sec.nice}) and also show that it is equivalent to our
previous conjecture (Section~\ref{sec.monst}). In order to obtain the nice
formula for~$D(n)$, we have to evaluate some related determinants
(Section~\ref{sec.lemmas}), which we then combine via the
Desnanot-Jacobi-Dodgson identity. In Section~\ref{sec.gen}, we identify these
determinants as special cases of some more general (infinite) families of
determinants and present several theorems and conjectures for their closed
forms.  All of them have a combinatorial meaning, as will be explained in
Section~\ref{sec.comb}. We first introduce the main object of study of
this article, the generalized determinant with shifted corner:
\begin{defn}\label{def.Dst}
For $n,s,t\in\Z$, $n\geq1$, and $\mu$ an indeterminate, we define $D_{s,t}(n)$
to be the following $(n\times n)$-determinant:
\[
  D_{s,t}(n) \colonequal \det_{\genfrac{}{}{0pt}{}{s \leq i < s+n}{t \leq j < t+n}}
  \left( \delta_{ij}+{ \binom{\mu+i+j-2}{j}}\right),  \qquad  n \geq 1.
\]
\end{defn}
Note that Andrews's determinant is a special case of it, namely
$D(n) = D_{1,1}(n)$, and that \eqref{eq.AndDet} equals $D_{0,0}(n)$ after
replacing $\mu$ by $\mu+2$.

\paragraph{Notation.}
We employ the usual notation $(x)_k$ for the Pochhammer symbol (also known as
rising factorial), that is defined as follows:
\[
  (x)_k \colonequal \begin{cases}
    x(x+1)\cdots(x+k-1), & k>0, \\
    1, & k=0, \\
    \frac{1}{(x+k)_{-k}}, & k<0.
  \end{cases}
\]
The short-hand notation $(x)_k^2$ is to be interpreted as
$\bigl((x)_k\bigr)^2$. The double factorial is defined, as usual, as
\[
  n!! \colonequal \begin{cases}
    2\cdot 4\cdots (n-2)\cdot n, & \text{if $n$ is even,} \\
    1\cdot 3\cdots (n-2)\cdot n, & \text{if $n$ is odd.}
  \end{cases}
\]

\section{Combinatorial Background}\label{sec.comb}

Before we go into details about the evaluations of the mentioned determinant
$D_{1,1}(n)$, and more generally $D_{s,t}(n)$, we want to give a combinatorial
interpretation of these determinants, namely we exhibit certain combinatorial
objects (rhombus tilings) that are counted by them.

The determinant~$D_{0,0}(n)$, which is given in~\eqref{eq.AndDet}, was
evaluated by George Andrews~\cite{Andrews79}, because it counts descending
plane partitions.  Christian Krattenthaler~\cite{Krattenthaler06} observed
that it equivalently counts cyclically symmetric rhombus tilings of a hexagon
with a triangular hole, where the size of the hole is related to the
parameter~$\mu$ \cite[Thm.~6]{CiucuEisenkoeblKrattenthalerZare01}.
From this, we deduce that our generalized version can count similar objects.
Throughout this section, we use the transformed parameter $\lambda:=\mu-2$,
which turns out to be more natural in this context (compare also with Andrews'
paper~\cite{Andrews79}).

The first observation is that $D_{s,t}(n)$ can be written as a sum of minors.
For this purpose, we rewrite it by performing index shifts on $i$ and~$j$:
\[
  D_{s,t}(n) = \det_{\genfrac{}{}{0pt}{}{s \leq i < s+n}{t \leq j < t+n}}
  \left( \delta_{ij}+{ \binom{\lambda+i+j}{j}}\right) =
  \det_{\genfrac{}{}{0pt}{}{1 \leq i \leq n}{1 \leq j \leq n}}
  \left( \delta_{i+s-t,j}+{ \binom{\lambda+i+j+s+t-2}{j+t-1}}\right).
\]
For the sake of readability, we abbreviate the latter binomial
coefficient by $b_{i,j}$, and do not denote the dependency on $s$ and~$t$.
Let $i\in\{1,\dots,n\}$ be such that $1\leq i+s-t\leq n$, i.e.\ the $i$-th
row contains one entry where the Kronecker delta evaluates to~$1$, then 
by Laplace expansion with respect to the $i$-th row one obtains
\[
  D_{s,t}(n) =
    \det_{\genfrac{}{}{0pt}{}{1 \leq i \leq n}{1 \leq j \leq n}}\bigl(\delta_{i+s-t,j}+b_{i,j}\bigr) =
    \sum_{j=1}^n (-1)^{j+1} \bigl(\delta_{i+s-t,j}+b_{i,j}\bigr) M^i_j =
    (-1)^{s-t} M^i_{i+s-t} + \sum_{j=1}^n (-1)^{i+j} b_{i,j} M^i_j,
\]
where $M^i_j$ denotes the $(i,j)$-minor of the corresponding matrix.
More generally, for any matrix~$A$, we can write
$\det(A)=\det(A^-)+(-1)^{i+j}M^i_j$, where $A^-$ denotes the matrix $A$ after
subtracting~$1$ from its $(i,j)$-entry. Applying this formula recursively
to the determinant~$D_{s,t}(n)$, until all $1$'s coming from the Kronecker deltas
are eliminated, yields the following identity
\begin{equation}\label{eq.SumOfMinors}
  D_{s,t}(n) =
  \begin{cases}
    \displaystyle\sum_{I\subseteq\{1,\dots,n-s+t\}} (-1)^{(s-t)\cdot|I|} \cdot \det\bigl(B^I_{I+s-t}\bigr),
      & \text{if } s\geq t,\\
    \displaystyle\sum_{I\subseteq\{1,\dots,n-t+s\}} (-1)^{(s-t)\cdot|I|} \cdot \det\bigl(B^{I+t-s}_I\bigr),
      & \text{if } s\leq t,
    \rule{0pt}{20pt}
  \end{cases}
\end{equation}
where $I+x=\{i+x \mid i\in I\}$ and where $B^I_J$ denotes the matrix that is
obtained by deleting all rows with indices in~$I$ and all columns with indices
in~$J$ from the matrix $B_{s,t}(n)=(b_{i,j})_{1\leq i,j\leq n}$. In other
words, we are summing over all subsets of positions where the Kronecker delta
evaluates to~$1$, and for each such subset we add or subtract the corresponding minor
$\det\bigl(B^I_J\bigr)$.

The second observation is that, by the Lindstr\"om--Gessel--Viennot
lemma~\cite{Lindstroem73,GesselViennot85}, $\det\bigl(B_{s,t}(n)\bigr)$ counts
$n$-tuples of non-intersecting paths in the integer lattice~$\N^2$: the start
points are $(\lambda+s,0)$, \mbox{$(\lambda+s+1,0)$}, \dots,
$(\lambda+s+n-1,0)$, the end points are $(0,t),(0,t+1),\dots,(0,t+n-1)$, and
the allowed steps are $(0,1)$ and $(-1,0)$; see Figure~\ref{fig.paths} (left)
for an example. The number of paths starting at $(\lambda+s+i-1,0)$ and ending
at $(0,t+j-1)$ is given by $\binom{\lambda+i+j+s+t-2}{j+t-1}$, which is
precisely the $(i,j)$-entry of~$B_{s,t}(n)$. Note that this counting is only
correct if $\lambda+s\geq0$; in the following we will assume that this
condition is satisfied. We do not know of a combinatorial interpretation when
$\lambda+s<0$.

\begin{figure}
  \begin{center}
    \includegraphics[width=0.5\textwidth]{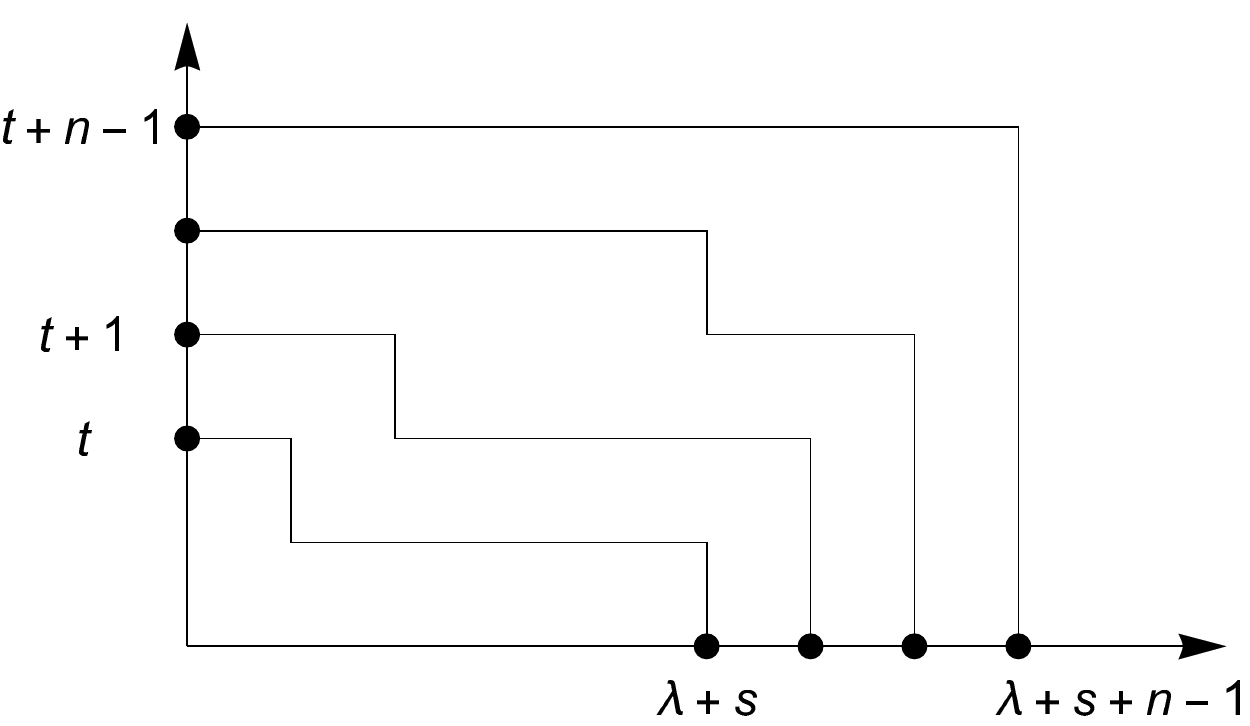}
    \qquad\qquad
    \includegraphics[width=0.3\textwidth]{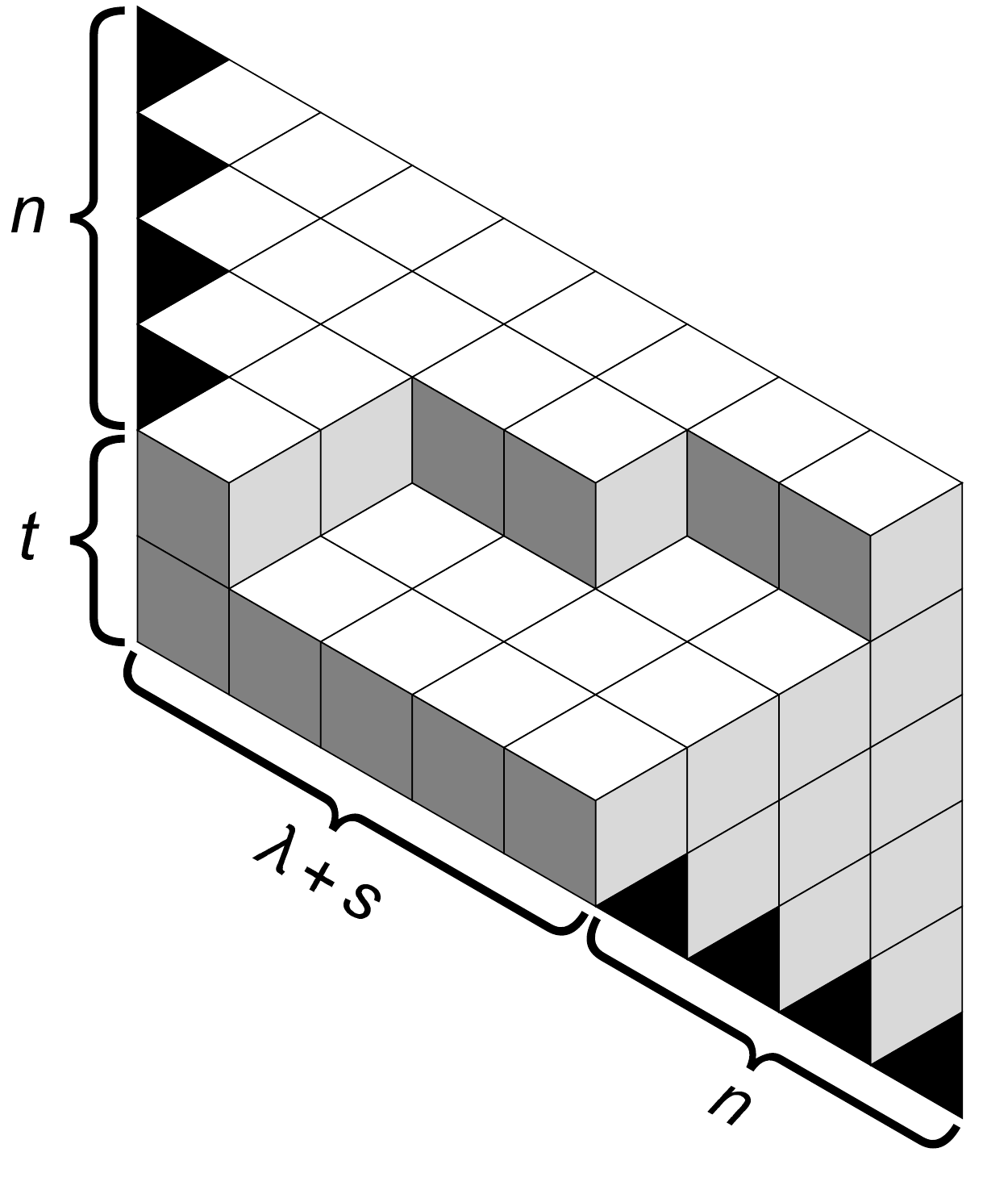}
  \end{center}
  \caption{A tuple of non-intersecting lattice paths (for $n=4$, $t=2$, and
    $\lambda+s=5$), and the corresponding rhombus tiling of a lozenge with some
    missing triangles (black): the white rhombi correspond to left-steps and
    the light-gray rhombi correspond to up-steps.}
  \label{fig.paths}
\end{figure}

If $|I|=|J|$ then $\det\bigl(B^I_J\bigr)$ counts the $(n-|I|)$-tuples of
non-intersecting paths where the start points with indices~$I$ and the end
points with indices~$J$ are omitted. In the case $s=t$, the expression
$\sum_{I\subseteq\{1,\dots,n\}}\det\bigl(B^I_I\bigr)$ counts all tuples of
non-intersecting paths for all subsets of start points (and the same subset
of end points). If $s>t$ then we use $\det\bigl(B^I_{I+s-t}\bigr)$ with
$I\subseteq\{1,\dots,n-s+t\}$. This means that we never omit the last $s-t$
start points on the horizontal axis and we never omit the first $s-t$ end
points on the vertical axis (counted from bottom to top). Moreover, the
omitted start and end points follow the same pattern, shifted by $s-t$.  If
$t>s$ then we never omit the first $t-s$ start points and the last $t-s$
end points.

The third and final observation is that the previously described
non-intersecting lattice paths are in bijection with rhombus tilings of a
lozenge-shaped region, where certain triangles on the border are cut out. They
correspond to the start and end points; see the right part of
Figure~\ref{fig.paths} where these triangles are colored black.  The two
types of steps (left and up) correspond to two orientations of the rhombi
(colored white and light-gray), while rhombi of the third possible
orientation (colored dark-gray) fill the areas which are not covered by
paths. From Figure~\ref{fig.paths} it is apparent that the lozenge has width
$\lambda+n+s$ and height $n+t$, and that $n$ black triangles are placed at the
right end of its lower side and another $n$ black triangles at the top of its
left vertical side.  From the bijection with lattice paths we see that the
number of rhombus tilings of such a lozenge is given by the determinant
$\det\bigl(B_{s,t}(n)\bigr)$.

In order to give a combinatorial interpretation to the determinant
$D_{s,t}(n)$, we have to sum up the counts of many similar tiling problems,
according to the sum of minors~\eqref{eq.SumOfMinors}. More precisely,
label the black triangles on the lower side of the lozenge with numbers from
$1$ to $n$ (from left to right), and similarly those on the vertical side
(from bottom to top). Then $\det\bigl(B^I_J\bigr)$ counts rhombus tilings of
the lozenge where all black triangles on the lower side with labels in~$I$ are
removed, and similarly, all black triangles on the vertical side with labels
in~$J$. Instead of adding up the results of many counting problems, we can
elegantly obtain the same result from a single counting problem, by
introducing cyclically symmetric rhombus tilings of hexagonal regions.

For this purpose, we rotate the lozenge by $120^\circ$ and by $240^\circ$ and
put the three copies together such that corresponding triangles share an
edge. We illustrate this procedure in Figure~\ref{fig.lozenges}: on the left
we show the three copies of the lozenge from Figure~\ref{fig.paths} with
parameters $s=4$, $t=2$, $n=4$, and $\lambda=2$. Since $s-t=2$ we never omit the
last two start points and the first two end points. Therefore, the
corresponding triangles are colored black. The fact that the remaining
start and end points may be omitted, is indicated by lighter colors. The
relation between $I$ and $J=I+s-t$ is visualized by matching colors: for two
triangles of the same color we have that either both are present or both are
omitted. The three copies of the lozenge are glued together such that
triangles of the same color share an edge. Note that this implies that none of
the black triangles will have a partner.

\begin{figure}
  \begin{center}
    \raisebox{20pt}{\includegraphics[width=0.33\textwidth]{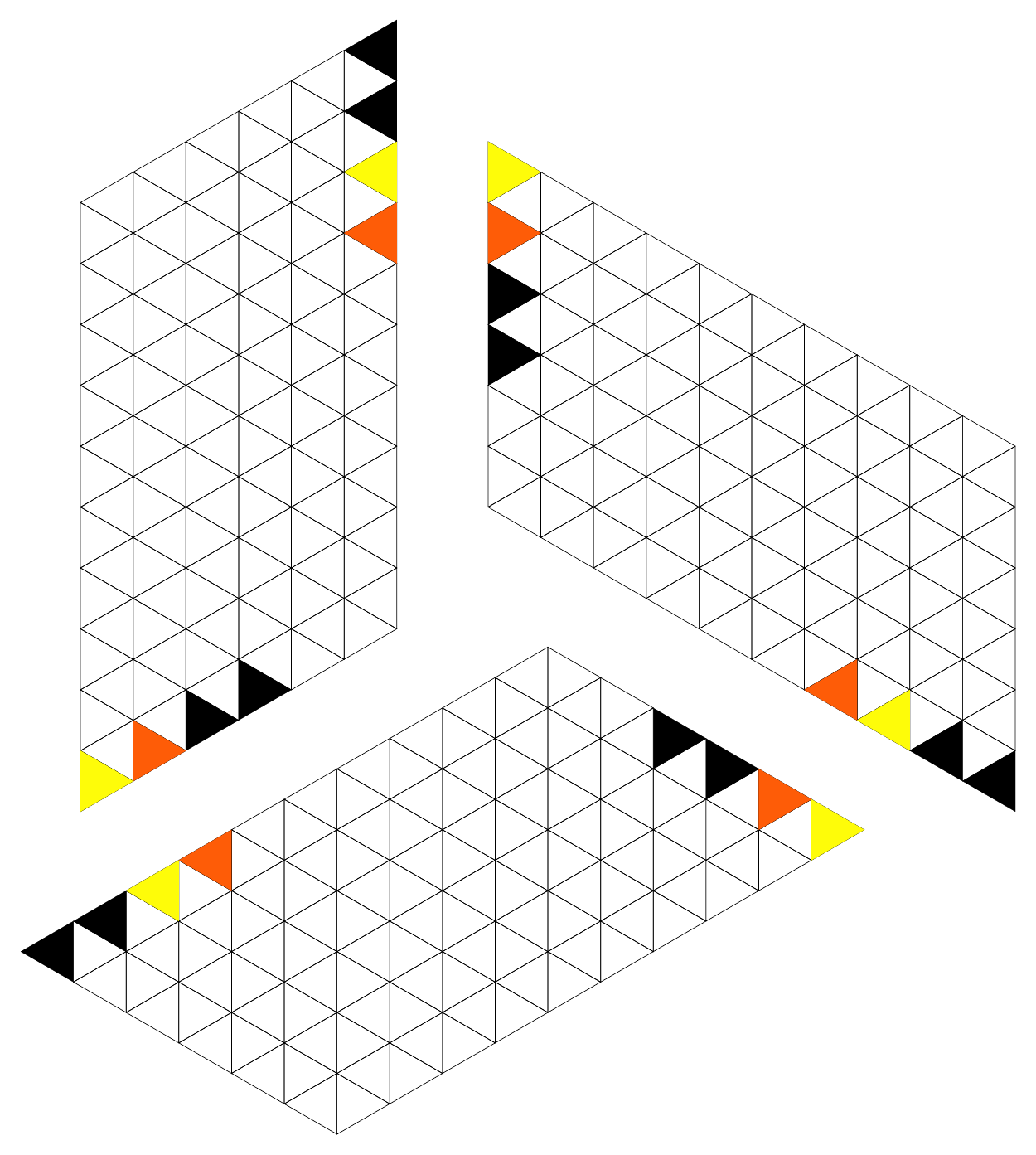}}
    \qquad\qquad
    \includegraphics[width=0.4\textwidth]{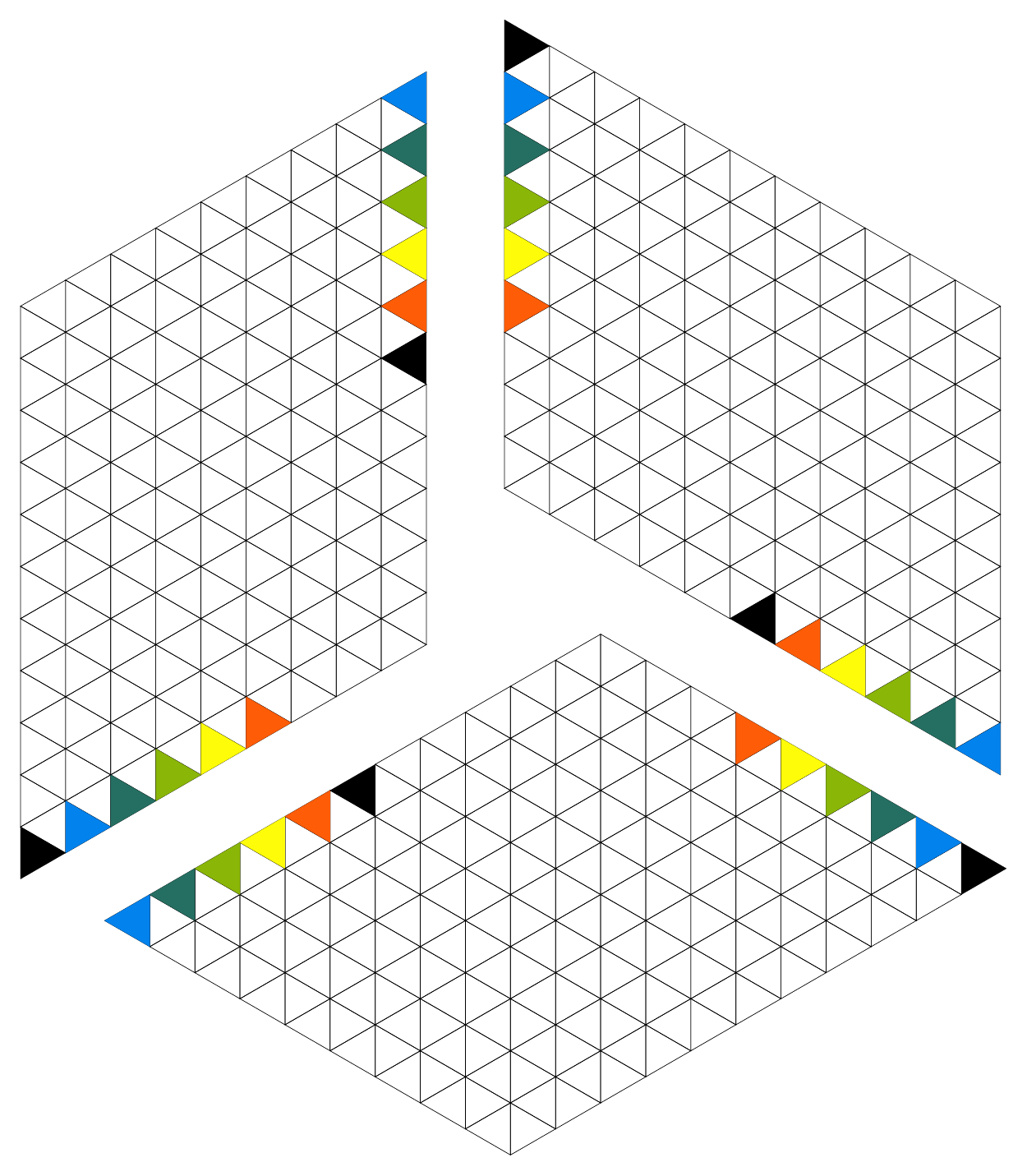}
  \end{center}
  \caption{Gluing together three copies of a lozenge; the left figure
    corresponds to the parameters $s=4$, $t=2$, $n=4$, $\lambda=2$, while the
    right figure has $s=2$, $t=3$, $n=6$, $\lambda=3$}
  \label{fig.lozenges}
\end{figure}

Now we obtain a region that is either a hexagon (if $s=t$) or that otherwise
has the shape of a pinwheel; see Figure~\ref{fig.pinwheel}.  In both cases,
there remains a ``hole'' in the center, except when $\lambda=0$. If
$\lambda\neq0$ then this hole has the shape of an equilateral triangle of side
length $|\lambda|$, pointing to the right if $\lambda>0$ and pointing to the
left if $\lambda<0$. We have to ensure that no rhombus crosses the border of
the original lozenge except for those positions that correspond to the start
and end points of the paths. For this reason, we place a ``border line'' of
length $\min(s,t)$ at each corner of the triangular hole and prohibit any
rhombus to lie across this border. Note that in the case $\lambda<0$ the
vertical border actually starts at the lower vertex of the (left pointing)
triangular hole, so that $\min(s,t,-\lambda)$ unit segments of the border
coincide with the right side of the triangular hole (and similarly for the
other two border lines). Each of these border lines is continued by $|s-t|$
unit triangular holes that point either in clockwise direction (if $s>t$) or
in counter-clockwise direction (if $s<t$). The same number of triangles
appears at the ``wings'' of the pinwheel, at a distance of $n-|s-t|$ from the
end of the border line; these triangles point in the opposite direction.

Since we have now three copies of the original domain, we have to avoid
overcounting: this is done by restricting the count to rhombus tilings that
are cyclically symmetric. At the same time this restriction automatically
ensures that the relation between start and end points is satisfied, namely
that they are distributed in the same manner, only shifted by $|s-t|$,
as described before.

By construction, we have obtained a region whose cyclically symmetric rhombus
tilings are counted by the determinant~$D_{s,t}(n)$, provided that $s-t$ is
even.  If $s-t$ is odd, the count is weighted by $+1$ and $-1$: the sign is
determined by the parity of the number of rhombi crossing the original
vertical side of the lozenge. Recall that the sign comes from
$(-1)^{(s-t)\cdot|I|}$ in~\eqref{eq.SumOfMinors}. The cardinality~$|I|$
corresponds to the number of vertical line segments between the two vertical
strips of black triangles that are ``visible'', i.e., that are not covered by a
horizontal rhombus. In other words: if there is an even number of line
segments that are not crossed by a horizontal rhombus then the count is
weighted with $+1$, otherwise with $-1$. By a ``horizontal rhombus'' we mean
one that is built of two triangles sharing a vertical edge.

\begin{figure}
  \begin{center}
    \includegraphics[width=0.4\textwidth]{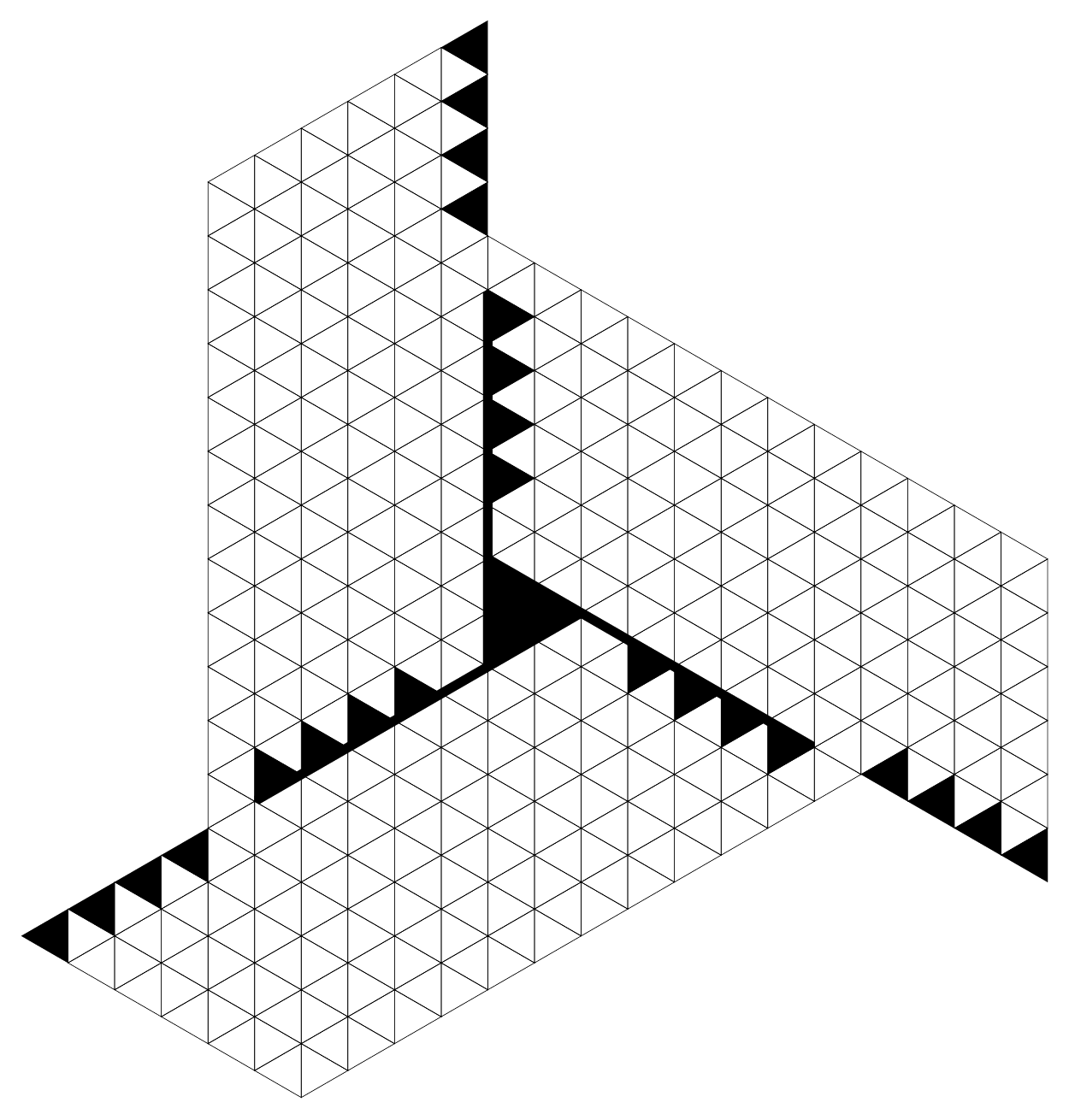}
    \qquad\qquad
    \includegraphics[width=0.4\textwidth]{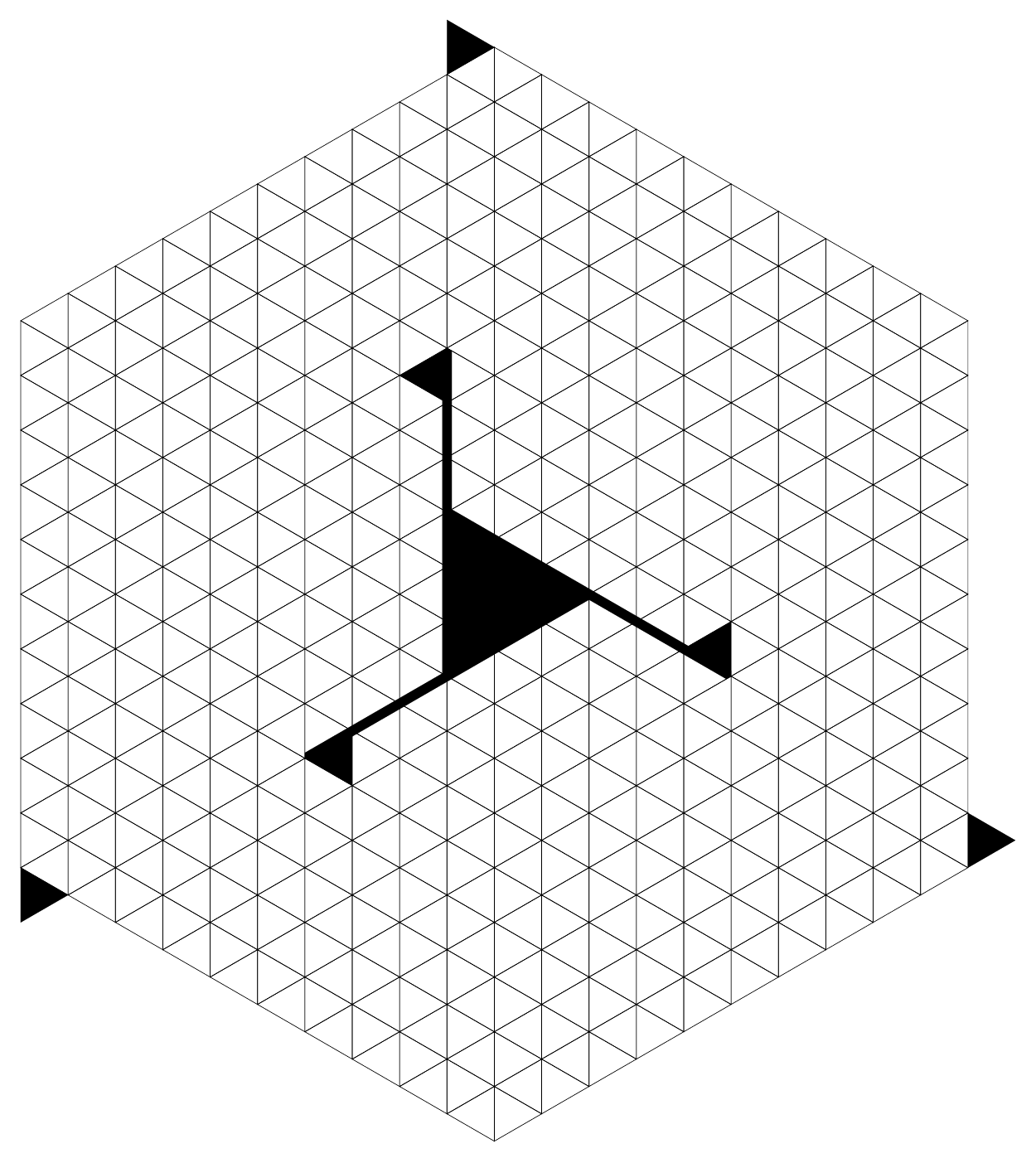}
  \end{center}
  \caption{Pinwheel-shaped regions with holes: the left figure corresponds to
    the parameters $s=5$, $t=1$, $n=5$, $\lambda=2$, the right figure corresponds
    to $s=2$, $t=3$, $n=6$, $\lambda=3$ (same as in Figure~\ref{fig.lozenges}).}
  \label{fig.pinwheel}
\end{figure}

The construction can be simplified by noting that a row of small triangular
holes induces a unique rhombus tiling when completing it to a big equilateral
triangle. Hence the pinwheel-shaped region can be replaced by a hexagon, by
cutting off three equilateral triangles of size $|s-t|$, without changing the
number of rhombus tilings. Similarly, the holes inside the region can be
re-interpreted as four triangular holes, of size $|\lambda|$ resp. $|s-t|$, that
are connected by boundary lines. We give an illustration of these regions in
Figure~\ref{fig.hexagon}.

\begin{figure}
  \begin{center}
    \includegraphics[width=0.4\textwidth]{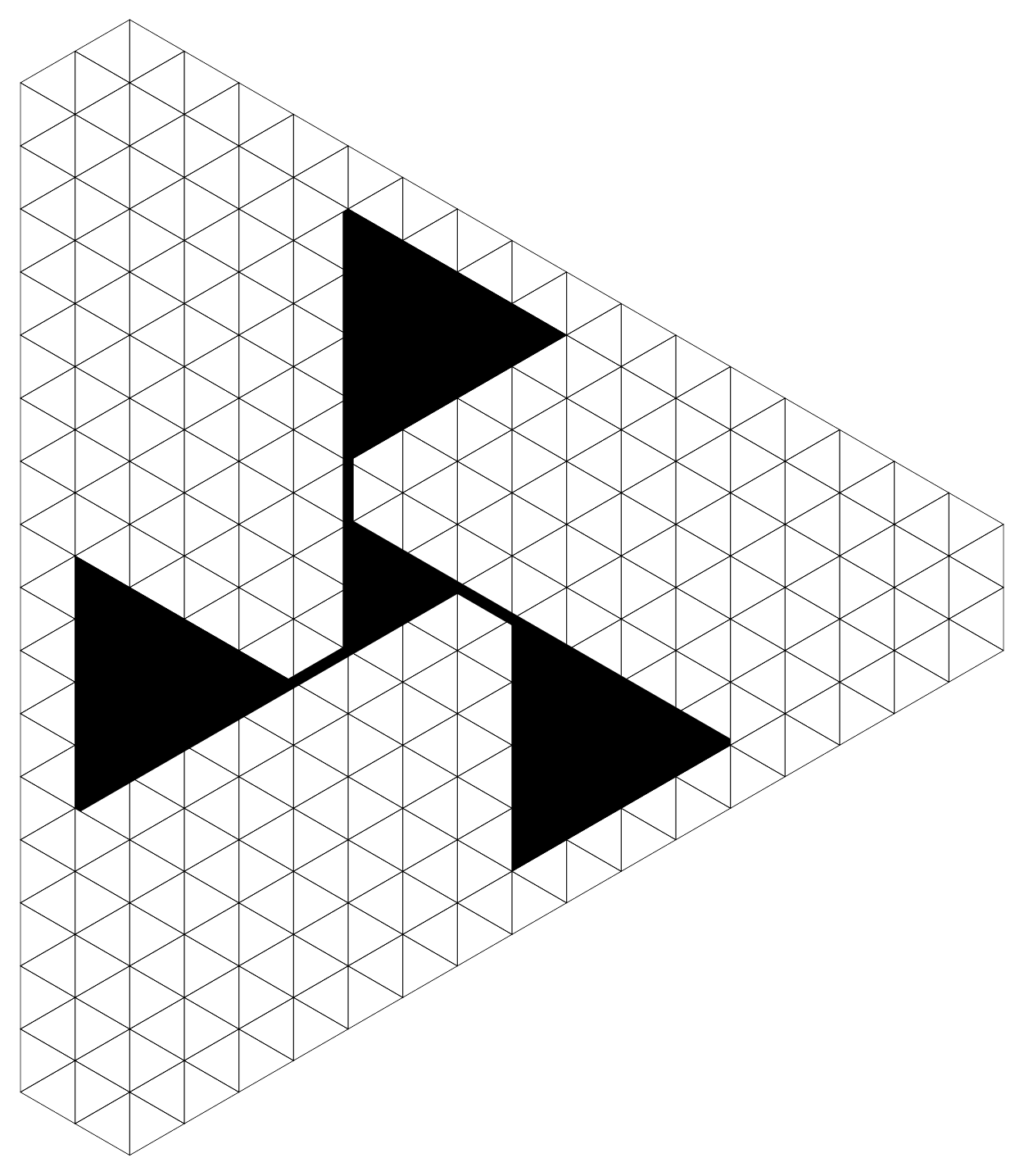}
    \qquad\qquad
    \includegraphics[width=0.4\textwidth]{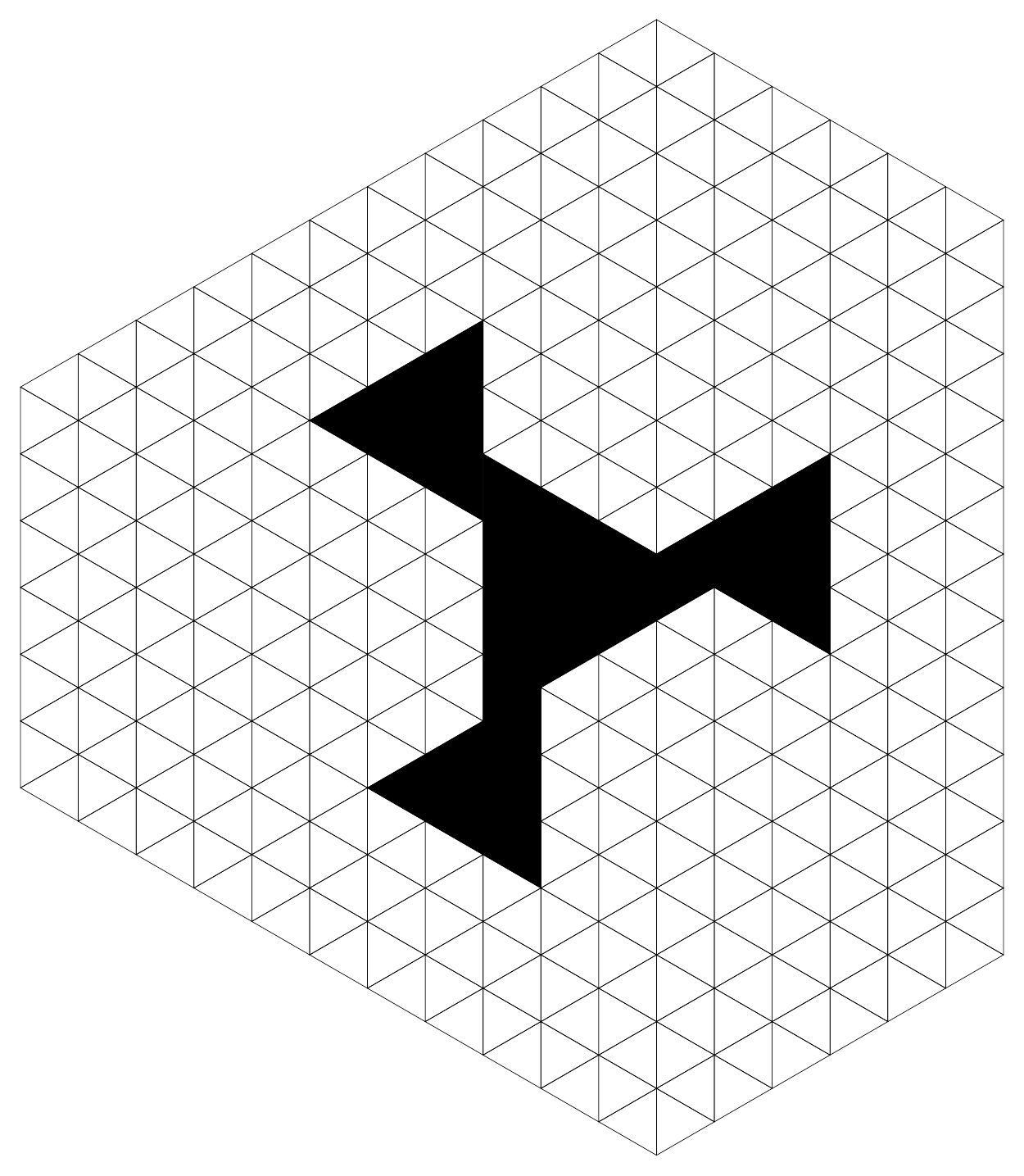}
  \end{center}
  \caption{Hexagonal regions with big triangular holes and border lines: the
    left figure corresponds to the same parameters as in
    Figure~\ref{fig.pinwheel} ($s=5$, $t=1$, $n=5$, $\lambda=2$), the right figure
    corresponds to $s=-1$, $t=2$, $n=6$, $\lambda=4$.}
  \label{fig.hexagon}
\end{figure}

As an example, we have worked out all cyclically symmetric rhombus tilings of
the hexagon that corresponds to $D_{1,1}(2)$ with $\lambda=1$; see
Figure~\ref{fig.tilings}. In this case, one can easily calculate
\[
  D_{1,1}(2)\big|_{\lambda\to1}
  = \begin{vmatrix} 4 & 6 \\ 4 & 11 \end{vmatrix}
  = 20.
\]

Another example that illustrates our combinatorial construction is the
identity
\[
  D_{s,t}(n) = D_{t+\lambda,s+\lambda}(n)\big|_{\lambda\to-\lambda}
\]
that follows directly by the mirror symmetry of the underlying tiling
regions. Assuming $\lambda\geq0$, the determinant $D_{s,t}(n)$ counts cyclically
symmetric rhombus tilings of a hexagon that has a triangular hole of size
$\lambda$ pointing to the right, with border lines of length $\min(s,t)$, to
each of which another triangular hole of size $|s-t|$ is attached, pointing
in clockwise direction if $s>t$. When we consider the transformed parameters
$s'=t+\lambda$, $t'=s+\lambda$, and $\lambda'=-\lambda$, we obtain a hexagonal region
with a hole of size $|\lambda'|=\lambda$ pointing to the left, with
border lines of length $\min(s',t')=\min(s,t)+\lambda$, each of which shares
a segment of length $\lambda$ with the hole (so only $\min(s,t)$ units are visible),
and with three other triangular holes of size $|s'-t'|=|s-t|$ each, pointing
in counterclockwise direction if $t'>s'$ (${\iff}$ $s>t$). Thus these two regions
are symmetric w.r.t.\ to a vertical axis and therefore possess the same
number of rhombus tilings.

\begin{figure}
  \begin{center}
    \includegraphics[width=0.8\textwidth]{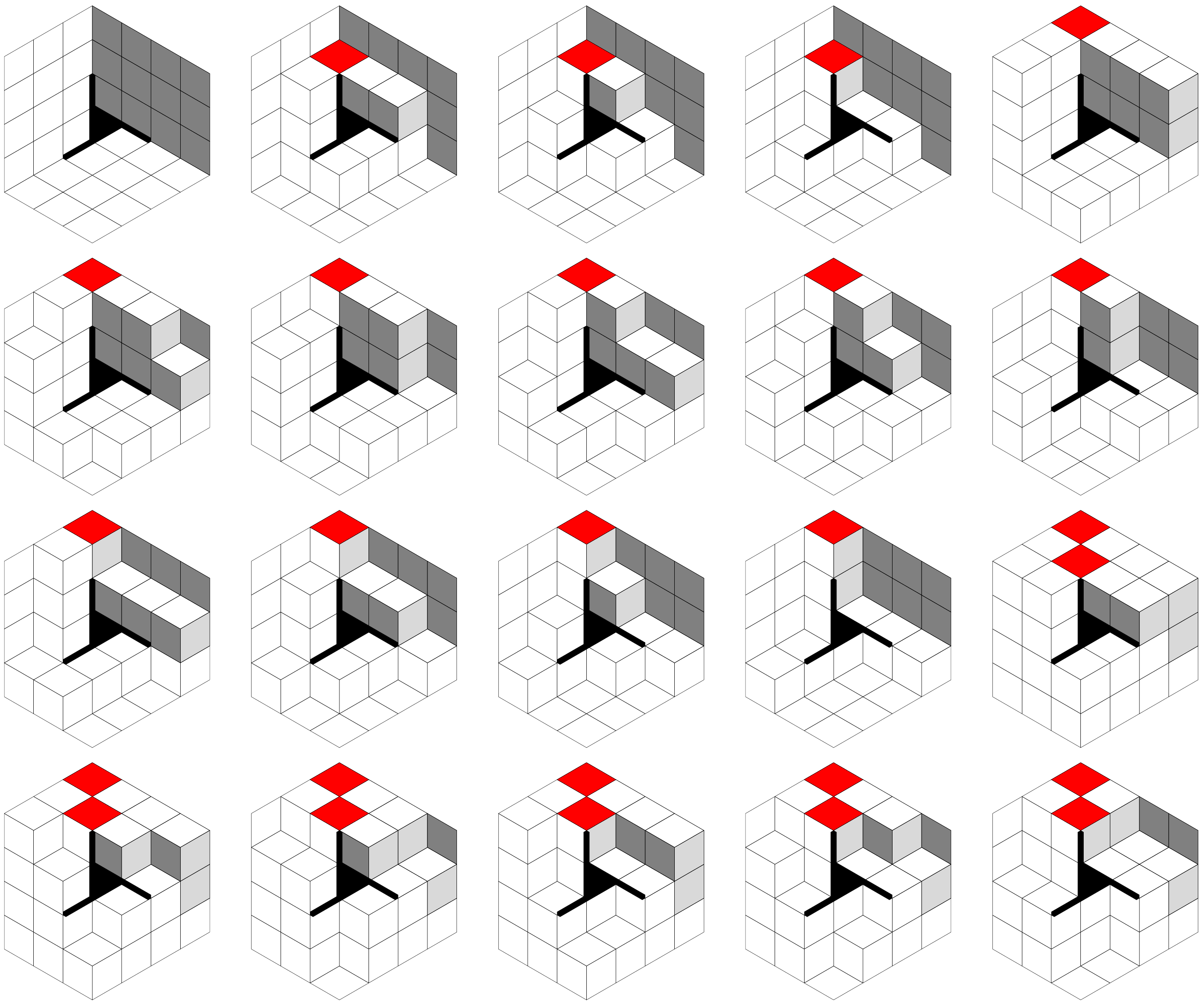}
  \end{center}
  \caption{All 20 cyclically symmetric rhombus tilings for the parameters
    $s=t=1$, $n=2$, and $\lambda=1$. The original lozenge is highlighted by
    shaded rhombi, the horizontal rhombi marking the end points of the
    lattice paths are colored red.}
  \label{fig.tilings}
\end{figure}

\section{Related Determinants}\label{sec.lemmas}

In this section, we prove a few easier results about particular instances of
the determinant~$D_{s,t}(n)$, with specific shifted corners, by using computer
proofs. Later, we put all these results together and obtain from it a
``closed-form'' formula for~$D(n)\quad (=D_{1,1}(n))$, via the celebrated
Desnanot-Jacobi-Dodgson identity: let $\bigl(m_{i,j}\bigr)_{i,j\in\Z}$ be a doubly
infinite sequence and denote by $M_{s,t}(n)$ the determinant of the $(n\times
n)$-matrix whose upper left entry is at $m_{s,t}$, more precisely the matrix
$\bigl(m_{i,j}\bigr)_{s\leq i<s+n,t\leq j<t+n}$. Then:
\begin{equation} \label{eq.DJD}
  M_{s,t}(n)M_{s+1,t+1}(n-2) = M_{s,t}(n-1)M_{s+1,t+1}(n-1)- M_{s+1,t}(n-1)M_{s,t+1}(n-1). \tag{DJD}
\end{equation}
For an excellent overview of this topic see~\cite{AmdeberhanZeilberger01}.

The following result was conjectured in~\cite{Andrews80a}, and in 2013 it was
proven by the authors of the present paper~\cite[Thm.~1]{KoutschanThanatipanonda13}:
\begin{thm} \label{thm.KT}
Let the determinant $D_{1,1}(n)$ be as in Definition~\ref{def.Dst}.
Then the following equation holds:
\begin{align*}
  \frac{D_{1,1}(2n)}{D_{1,1}(2n-1)} &=
   (-1)^{(n-1)(n-2)/2} \, 2^n \, \frac{
     \left(\frac{\mu}{2}+2n+\frac{1}{2}\right)_{\!n-1} \left(\frac{\mu}{2}+n\right)_{\!\lfloor (n+1)/2\rfloor}}{
     \bigl(n\bigr)_{\!n} \left(-\frac{\mu}{2}-2n+\frac{3}{2}\right)_{\!\lfloor (n-1)/2\rfloor}} \\
  &= 2^n \, \frac{
     \left(\frac{\mu}{2}+2n+\frac{1}{2}\right)_{\!n-1} \left(\frac{\mu}{2}+n\right)_{\!\lfloor (n+1)/2\rfloor}}{
     \bigl(n\bigr)_{\!n} \left(\frac{\mu}{2}+\bigl\lfloor\frac{3n}{2}\bigr\rfloor+\frac{1}{2}\right)_{\!\lfloor (n-1)/2\rfloor}} \\
  &= \frac{\bigl(\mu+2n\bigr)_{\!n} \left(\frac{\mu}{2}+2n+\frac{1}{2}\right)_{\!n-1}}{
           \bigl(n\bigr)_{\!n} \left(\frac{\mu}{2}+n+\frac{1}{2}\right)_{\!n-1}}.
\end{align*}
\end{thm}

In the following we state five lemmas with computer proofs, concerning
special cases of the general determinant~$D_{s,t}(n)$. They are employed
afterwards to obtain closed-form formulas for the determinants $D_{0,0}(n)$,
$D_{1,0}(n)$ and $D_{0,1}(n)$; see Propositions \ref{prop.D00},
\ref{prop.D10}, and \ref{prop.D01}, respectively. These in turn will be used
in the main formula for~$D_{1,1}(n)$ in Section~\ref{sec.nice}.

\begin{lem} \label{lem.E10}
$D_{1,0}(2n)=0$ for all integers $n\geq1$.
\end{lem}
\begin{proof}
In order to prove that the determinant vanishes, we exhibit a concrete nontrivial
linear combination of the columns of the matrix:
\newcommand{\vph}{\vphantom{\binom{\mu+2n-2}{2n-1}}}
\[
  c_{n,1} \cdot \begin{pmatrix} \binom{\mu-1}{0}\vph \\[1ex] \binom{\mu}{0}\vph \\
    \vdots \\[1ex] \binom{\mu+2n-3}{0}\vph \\[1ex] \binom{\mu+2n-2}{0}\vph \end{pmatrix} +
  c_{n,2} \cdot \begin{pmatrix} \binom{\mu}{1}+1\vph \\[1ex] \binom{\mu+1}{1}\vph \\
    \vdots \\[1ex] \binom{\mu+2n-2}{1}\vph \\[1ex] \binom{\mu+2n-1}{1}\vph \end{pmatrix} +
  \dots +
  c_{n,2n} \cdot \begin{pmatrix} \binom{\mu+2n-2}{2n-1} \\[1ex] \binom{\mu+2n-1}{2n-1} \\
    \vdots \\[1ex] \binom{\mu+4n-4}{2n-1}+1 \\[1ex] \binom{\mu+4n-3}{2n-1} \end{pmatrix}
  = \begin{pmatrix} 0\vph \\[1ex] 0\vph \\
    \vdots \\[1ex] 0\vph \\[1ex] 0\vph \end{pmatrix},
\]
where the coefficients $c_{n,j}$ are rational functions in $\Q(\mu)$.  For all
$n\leq30$ the nullspace of $D_{1,0}(2n)$ has dimension~$1$, and it seems
likely that this is the case for all~$n$. However, we need not care whether
this is true or not, the important fact is that the coefficients $c_{n,j}$ for
$n\leq30$ and $1\leq j\leq 2n$ are determined uniquely if we impose
$c_{n,2n}=1$. Hence they are easily computed by linear algebra, and we can use
these explicitly computed values to construct recurrence equations satisfied
by them (colloquially called ``guessing''). Now we consider the infinite
sequence $\bigl(c_{n,j}\bigr){}_{n,j\in\N}$ that is defined by these
recurrence equations, subject to initial conditions that agree with the
explicitly computed $c_{n,j}$. We want to show that for all~$n$ the vector
$\bigl(c_{n,j}\bigr){}_{1\leq j\leq 2n}$ lies in the kernel of~$D_{1,0}(2n)$
(so far we only know this for $n$ up to~$30$). This reduces to proving the
holonomic function identity
\[
  \sum_{j=1}^{2n} \binom{\mu+i+j-3}{j-1} c_{n,j} = -c_{n,i+1} \qquad (1\leq i\leq 2n).
\]
Using the computer algebra package HolonomicFunctions~\cite{Koutschan10b},
developed by the first-named author, it can be proven without much effort.
The details of the computer calculations can be found in~\cite{EM}.
\end{proof}

\begin{lem} \label{lem.E01}
$D_{0,1}(2n)=0$ for all integers $n\geq1$.
\end{lem}
\begin{proof}
The proof is analogous to the one of Lemma~\ref{lem.E10}. The detailed
computations can be found in the electronic material~\cite{EM}.
\end{proof}

\begin{lem} \label{lem.R00}
\[
  \frac{D_{0,0}(2n)}{D_{0,0}(2n-1)} =
  \frac{
    \bigl(\mu+2n-2\bigr)_{\!n-1}
    \left(\frac{\mu}{2}+2n-\frac{1}{2}\right)_{\!n}
  }{
    \bigl(n\bigr)_{\!n}
    \left(\frac{\mu}{2}+n-\frac{1}{2}\right)_{\!n-1}
  }.
\]
\end{lem}
\begin{proof}
Note that $D_{0,0}(n)$ is basically the same determinant as~\eqref{eq.AndDet}
(upon replacing $\mu$ by $\mu+2$). Its evaluation was first achieved by George
Andrews~\cite{Andrews79}. The above statement is a corollary of his result, so
there is nothing to prove. Just for completeness, and to show that all
statements presented here can be treated with the same uniform approach, we
give also a computer algebra proof in~\cite{EM}.
\end{proof}

\begin{lem} \label{lem.R20}
\[
  \frac{D_{2,0}(2n)}{D_{2,0}(2n-1)} =
  \frac{
    \bigl(\mu+2n+1\bigr)_{\!n-1}
    \left(\frac{\mu}{2}+2n+\frac{1}{2}\right)_{\!n-1}
  }{
    \bigl(n\bigr)_{\!n-1}
    \left(\frac{\mu}{2}+n+\frac{1}{2}\right)_{\!n-1}
  }.
\]
\end{lem}
\begin{proof}
We employ computer algebra methods to prove the statement, following
Zeilberger's holonomic ansatz~\cite{Zeilberger07}. The overall proof strategy
is similar to the one in Lemma~\ref{lem.E10}: using an ansatz with
undetermined coefficients (``guessing'') we find the holonomic description of
an auxiliary bivariate sequence~$\bigl(c_{n,j}\bigr){}_{n,j\in\N}$ that
certifies the correctness of the statement. In contrast to
Lemma~\ref{lem.E10}, the statement we want to prove implies that the
determinant $D_{2,0}(2n)$ is nonzero, and hence we shall not succeed in
finding a nonzero vector in the nullspace of the corresponding
matrix. Instead, we delete its last row and consider the nullspace of the
obtained $(2n-1)\times(2n)$-matrix, and proceed as in the proof of
Lemma~\ref{lem.E10}: for concrete small~$n$ compute a vector of length~$2n$
that spans this (one-dimensional) nullspace, normalize it such that its last
component equals~$1$, and construct bivariate recurrence equations satisfied
by this data. This holonomic description (recurrences plus finitely many
initial values) uniquely defines an infinite
sequence~$\bigl(c_{n,j}\bigr){}_{n,j\in\N}$.  We use the HolonomicFunctions
package~\cite{Koutschan10b} to prove some general properties and identities of
this sequence.

First, we show that $c_{n,2n}=1$ holds for all $n\in\N$, by constructing a
linear combination of our recurrences (and possibly their shifted versions) in
which only terms of the form $c_{n,j},c_{n+1,j+2},c_{n+2,j+4},\dots$
occur. Substituting $j=2n$ yields a recurrence for the univariate sequence
$g_n:=c_{n,2n}$ and we can verify that the constant $1$ sequence is among its
solutions.

Second, we prove the following summation identity, where we denote by
$a_{i,j}$ the $(i,j)$-entry of $D_{2,0}(2n)$:
\[
  \sum_{j=1}^{2n} a_{i,j} \, c_{n,j} = 0, \qquad \text{for all } n\in\N \text{ and } 1\leq i\leq 2n-1.
\]
It follows by linear algebra that $c_{n,j}$ is closely related to the
$(2n,j)$-minor $M_{2n,j}$ of the matrix of $D_{2,0}(2n)$:
\[
  c_{n,j} = (-1)^{2n+j} \frac{M_{2n,j}}{M_{2n,2n}} = (-1)^j \frac{M_{2n,j}}{D_{2,0}(2n-1)}.
\]

Third, one observes that the $c_{n,j}$ with $1\leq j\leq 2n$ are the cofactors
of the Laplace expansion of $D_{2,0}(2n)$ with respect to the last row,
divided by $D_{2,0}(2n-1)$, which implies that
\[
  \sum_{j=1}^{2n} a_{2n,j} c_{n,j} = \frac{D_{2,0}(2n)}{D_{2,0}(2n-1)}.
\]
Hence, the proof is concluded by proving that this sum equals the asserted
quotient of Pochhammer symbols. The proofs of the summation identities are
carried out with HolonomicFunctions, and the details of these computations are
contained in the electronic material~\cite{EM}.
\end{proof}

\begin{lem} \label{lem.R02}
\[
  \frac{D_{0,2}(2n)}{D_{0,2}(2n-1)} =
  \frac{
    (2n-1) \, \bigl(\mu+2n-2\bigr)_{\!n+2}
    \left(\frac{\mu}{2}+2n+\frac{1}{2}\right)_{\!n-1}
  }{
    (\mu+2n) \, \bigl(n\bigr)_{\!n+2}
    \left(\frac{\mu}{2}+n+\frac{1}{2}\right)_{\!n-1}
  }.
\]
\end{lem}
\begin{proof}
The proof is analogous to that of Lemma~\ref{lem.R20}; details can be found
in~\cite{EM}.  However, we want to point out one issue that we encountered in
the computations: In the guessing step we had to omit some of the data, as it
was inconsistent with the rest of the data. More concretely, the recurrences
we found were not valid for $c_{n,j}$ at $n=1$.  For the rest of the proof,
this is irrelevant, but being unaware of this issue, one could get the
impression that no recurrences exist at all. This phenomenon is explained by
the fact that for $n=1$ the Kronecker delta does not appear in the matrix, and
hence this case is somehow special. (For the same reason, we have the
condition $n\geq r$ in Corollaries~\ref{cor.famE} and~\ref{cor.famF}, for
example.)
\end{proof}

\begin{prop} \label{prop.D00}
We have $D_{0,0}(n) = 2\displaystyle\prod_{i=1}^{n-1} R_{0,0}(i)$, in other
words $R_{0,0}(n) = D_{0,0}(n+1) / D_{0,0}(n)$, where
\begin{align*}
  R_{0,0}(2n) &=
  \frac{
    \bigl(\mu+2n\bigr)_{\!n}
    \left(\frac{\mu}{2}+2n+\frac{1}{2}\right)_{\!n-1}
  }{
    \bigl(n\bigr)_{\!n}
    \left(\frac{\mu}{2}+n+\frac{1}{2}\right)_{\!n-1}
  }, \\[1ex]
  R_{0,0}(2n-1) &=
  \frac{
    \bigl(\mu+2n-2\bigr)_{\!n-1}
    \left(\frac{\mu}{2}+2n-\frac{1}{2}\right)_{\!n}
  }{
    \bigl(n\bigr)_{\!n}
    \left(\frac{\mu}{2}+n-\frac{1}{2}\right)_{\!n-1}
  }.
\end{align*}
\end{prop}
\begin{proof}
Recall that this determinant is due to George Andrews~\cite{Andrews79}. In order
to put it into our context, we give an alternative proof.
If $n$ is even, we apply the Desnanot-Jacobi-Dodgson identity~\eqref{eq.DJD} to get 
\begin{align*}
  D_{0,0}(n+1)D_{1,1}(n-1) &= D_{0,0}(n)D_{1,1}(n) -\cancelto{0}{D_{0,1}(n)}\cancelto{0}{D_{1,0}(n)}, \\
  \frac{D_{0,0}(n+1)}{D_{0,0}(n)} &= \frac{D_{1,1}(n)}{D_{1,1}(n-1)},
\end{align*}
from which the claimed formula follows by using Theorem~\ref{thm.KT}.
The claims, $D_{0,1}(n) =0$ and $D_{1,0}(n) =0$, were stated in Lemma
\ref{lem.E10} and Lemma \ref{lem.E01}.
If $n$ is odd, the result is a direct consequence of Lemma~\ref{lem.R00}.
For the product formula, note that $D_{0,0}(1)=2$.
\end{proof}

\begin{prop} \label{prop.D10}
We have $D_{1,0}(2n+1) / D_{1,0}(2n-1) = R_{1,0}(n)$ where
\[
  R_{1,0}(n) \colonequal
  -\frac{
    \bigl(\mu+2n\bigr)_{\!n} \,
    \bigl(\mu+2n+1\bigr)_{\!n-1}
    \left(\frac{\mu}{2}+2n+\frac{1}{2}\right)^{\!2}_{\!n-1}
  }{
    \bigl(n\bigr)_{\!n} \,
    \bigl(n\bigr)_{\!n-1}
    \left(\frac{\mu}{2}+n+\frac{1}{2}\right)^{\!2}_{\!n-1}
  }.
\]
Moreover,
\[
  D_{1,0}(n) = \begin{cases}
    0, & \text{if $n$ is even}, \\
    \prod_{i=1}^{(n-1)/2} R_{1,0}(i), & \text{if $n$ is odd}. 
  \end{cases}
\]
\end{prop}
\begin{proof}
By applying~\eqref{eq.DJD} twice we obtain
\begin{align*}
  D_{1,0}(2n+1)D_{2,1}(2n-1) &= \cancelto{0}{D_{1,0}(2n)}D_{2,1}(2n) -D_{1,1}(2n)D_{2,0}(2n), \\
  \cancelto{0}{D_{1,0}(2n)}{D_{2,1}(2n-2)} &= D_{1,0}(2n-1)D_{2,1}(2n-1) -D_{1,1}(2n-1)D_{2,0}(2n-1).
\end{align*}
We then combine these two equations to get
\[
  \frac{D_{1,0}(2n+1)}{D_{1,0}(2n-1)}
  = -\frac{D_{1,1}(2n)}{D_{1,1}(2n-1)} \cdot \frac{D_{2,0}(2n)}{D_{2,0}(2n-1)},
\]
from which the formula for $R_{1,0}(n)$ follows, by invoking
Theorem~\ref{thm.KT} and Lemma~\ref{lem.R20}. The fact $D_{1,0}(2n)=0$ was
already stated in Lemma~\ref{lem.E10}.
\end{proof}

\begin{prop} \label{prop.D01}
We have $D_{0,1}(2n+1) / D_{0,1}(2n-1) = R_{0,1}(n)$ where
\[
  R_{0,1}(n) \colonequal
  -\frac{
    \bigl(\mu+2n-2\bigr)_{\!n+2} \,
    \bigl(\mu+2n+1\bigr)_{\!n-1}
    \left(\frac{\mu}{2}+2n+\frac{1}{2}\right)^{\!2}_{\!n-1}
  }{
    \bigl(n\bigr)_{\!n+2} \,
    \bigl(n\bigr)_{\!n-1}
    \left(\frac{\mu}{2}+n+\frac{1}{2}\right)^{\!2}_{\!n-1}
  }.
\]
Moreover,
\[
  D_{0,1}(n) = \begin{cases}
    0, & \text{if $n$ is even}, \\
    (\mu-1)\cdot\prod_{i=1}^{(n-1)/2} R_{0,1}(i), & \text{if $n$ is odd}. 
  \end{cases}
\]
\end{prop}
\begin{proof}
By applying~\eqref{eq.DJD} twice we obtain
\begin{align*}
  D_{0,1}(2n+1)D_{1,2}(2n-1) &= \cancelto{0}{D_{0,1}(2n)}D_{1,2}(2n) -D_{1,1}(2n)D_{0,2}(2n), \\
  \cancelto{0}{D_{0,1}(2n)}{D_{1,2}(2n-2)} &= D_{0,1}(2n-1)D_{1,2}(2n-1) -D_{1,1}(2n-1)D_{0,2}(2n-1).
\end{align*}
As before, we combine these two equations to get
\[
  \frac{D_{0,1}(2n+1)}{D_{0,1}(2n-1)}
  =-\frac{D_{1,1}(2n)}{D_{1,1}(2n-1)} \cdot \frac{D_{0,2}(2n)}{D_{0,2}(2n-1)},
\]
from which the formula for $R_{0,1}(n)$ follows, by invoking
Theorem~\ref{thm.KT} and Lemma~\ref{lem.R02}. The fact $D_{0,1}(2n)=0$ was
already stated in Lemma~\ref{lem.E01}. The product formula is obtained by
observing that $D_{0,1}(1)=\mu-1$.
\end{proof}

As an aside, we want to mention that our original plan was to use the quotient
of the two consecutive determinants $D_{-1,1}(2n+1)$ and $D_{-1,1}(2n)$, which
also factors nicely. However, we did not succeed in applying the holonomic
ansatz to solve this problem. More precisely, we were not able to guess a
holonomic description for the corresponding~$c_{n,j}$. Nevertheless, using our
other results, we can now state:
\begin{cor}\label{cor.Dm11}
\[
  \frac{D_{-1,1}(2n+1)}{D_{-1,1}(2n)} =
  \frac{
    (2n-1) \, \bigl(\mu+2n-2\bigr)_{\!n+2}
    \left(\frac{\mu}{2}+2n+\frac{1}{2}\right)_{\!n-1}
  }{
    (\mu+2n) \, \bigl(n\bigr)_{\!n+2}
    \left(\frac{\mu}{2}+n+\frac{1}{2}\right)_{\!n-1}
  }.
\]
\end{cor}
\begin{proof}
The assertion follows from
\[
  D_{-1,1}(2n+1)D_{0,2}(2n-1) = {D_{-1,1}(2n)}D_{0,2}(2n) -\cancelto{0}{D_{0,1}(2n)}D_{-1,2}(2n)
\]
by applying Proposition~\ref{prop.D01} and Lemma~\ref{lem.R02}.
\end{proof}

\section{Nice Closed Form for \texorpdfstring{$D_{1,1}(n)$}{D11(n)}}\label{sec.nice}

From Propositions~\ref{prop.D00}, \ref{prop.D10}, \ref{prop.D01} we have now
the values of $D_{0,0}(n), D_{1,0}(n)$, and $D_{0,1}(n)$ at our disposal, and
we will use them to derive, for the first time, a kind of a closed-form for
the mysterious determinant $D_{1,1}(n)$. In Figure~\ref{fig.tilings} it is
shown what kind of rhombus tilings are counted by~$D_{1,1}(n)$.  Once again,
we will use the Desnanot-Jacobi-Dodgson identity~\eqref{eq.DJD} (see
p.~\pageref{eq.DJD}) to glue the previous results together. By doing so, we
obtain a recurrence equation for~$D_{1,1}(n)$:
\[
  D_{0,0}(n)D_{1,1}(n-2) = D_{0,0}(n-1)D_{1,1}(n-1)-D_{1,0}(n-1)D_{0,1}(n-1).
\]
We replace $n$ with $n+1$, divide by $D_{0,0}(n)$, and apply Proposition~\ref{prop.D00}:
\[
  D_{1,1}(n) = R_{0,0}(n) D_{1,1}(n-1) + \frac{D_{1,0}(n)D_{0,1}(n)}{D_{0,0}(n)}.
\]
Since by Lemmas~\ref{lem.E01} and~\ref{lem.E10} $D_{0,1}(n) = D_{1,0}(n) = 0$
for even~$n$, the recurrence in this case simplifies:
\[
  D_{1,1}(n) = R_{0,0}(n) D_{1,1}(n-1) \qquad (n \text{ even}).
\]
For odd~$n$, using the Propositions~\ref{prop.D00}, \ref{prop.D10}, and \ref{prop.D01}, we obtain:
\begin{align*}
  D_{1,1}(n) &= R_{0,0}(n) D_{1,1}(n-1) + (\mu-1) \frac{\left(\prod_{j=1}^{(n-1)/2} R_{1,0}(j)\right)
    \left(\prod_{j=1}^{(n-1)/2} R_{0,1}(j)\right)}{2\prod_{j=1}^{n-1} R_{0,0}(j)} \\
  &= R_{0,0}(n) D_{1,1}(n-1) + \frac{(\mu-1)}{2} \prod_{j=1}^{(n-1)/2} \frac{R_{1,0}(j) R_{0,1}(j)}{R_{0,0}(2j-1)R_{0,0}(2j)}
  \qquad (n \text{ odd}).
\end{align*}
Splitting $R_{0,0}(i)$ into even and odd is reasonable, since it is anyway
defined differently for these cases.  Now, by unrolling the recurrence, we get
a ``closed form'', namely an explicit single sum expression, for $D_{1,1}(n)$:
\begin{equation}\label{eq.cf1}
  D_{1,1}(n) = \prod_{j=1}^n R_{0,0}(j) + \frac{(\mu-1)}{2} \sum_{k=1}^{\lfloor(n+1)/2 \rfloor}
  \left(\prod_{j=2k}^n R_{0,0}(j)\right) \left(\prod_{j=1}^{k-1} \frac{R_{1,0}(j) R_{0,1}(j)}{R_{0,0}(2j-1)R_{0,0}(2j)} \right)
\end{equation}

\begin{lem} \label{lem.PR1001}
\[
  \prod_{j=1}^{k-1} \frac{R_{1,0}(j) R_{0,1}(j)}{R_{0,0}(2j-1)R_{0,0}(2j)} =
  \frac{\bigl(\mu\bigr)_{\!3k-3}}{(2k-1)! \, \left(\frac{\mu}{2}+k-\frac{1}{2}\right)_{\!k-1}}
  \left(\prod_{j=1}^{k-1}
    \frac{\bigl(\mu+2j+1\bigr)_{\!j-1} \left(\frac{\mu}{2}+2j+\frac{1}{2}\right)_{\!j-1}}
         {\bigl(j\bigr)_{\!j-1} \left(\frac{\mu}{2}+j+\frac{1}{2}\right)_{\!j-1}}
  \right)^{\!\!2}.
\]
\end{lem}
\begin{proof}
First, we investigate the factor inside the product:
\begin{align*}
  & \kern-20pt \frac{R_{1,0}(j) R_{0,1}(j)}{R_{0,0}(2j-1)R_{0,0}(2j)} = {}\\
  &= \frac{
    \bigl(j\bigr)_{\!j} \,
    \bigl(\mu+2j-2\bigr)_{\!j+2} \,
    \bigl(\mu+2j+1\bigr)^{\!2}_{\!j-1}
    \left(\frac{\mu}{2}+j-\frac{1}{2}\right)_{\!j-1}
    \left(\frac{\mu}{2}+2j+\frac{1}{2}\right)^{\!3}_{\!j-1}
  }{
    \bigl(j\bigr)^{\!2}_{\!j-1} \,
    \bigl(j\bigr)_{\!j+2} \,
    \bigl(\mu+2j-2\bigr)_{\!j-1}
    \left(\frac{\mu}{2}+j+\frac{1}{2}\right)^{\!3}_{\!j-1}
    \left(\frac{\mu }{2}+2j-\frac{1}{2}\right)_{\!j}
  }
  \\
  &= \frac{
    (\mu+2j-1) (\mu+3j-3) (\mu+3j-2) (\mu+3j-1) \,
    \bigl(\mu+2j+1\bigr)^{\!2}_{\!j-1}
    \left(\frac{\mu}{2}+2j+\frac{1}{2}\right)^{\!2}_{\!j-1}
  }{
    j\, (2j+1) (\mu+4j-3) (\mu+4j-1) \,
    \bigl(j\bigr)^{\!2}_{\!j-1}
    \left(\frac{\mu}{2}+j+\frac{1}{2}\right)^{\!2}_{\!j-1}
  }
\end{align*}
By taking the product of this last expression, we get the asserted formula.
\end{proof}

\begin{thm} \label{thm.D11}
Let $\mu$ be an indeterminate and let $D_{1,1}$ be defined as in
Definition~\ref{def.Dst}. Let $\rho_k$ be defined as $\rho_0(a,b)=a$ and
$\rho_k(a,b)=b$ for $k>0$. If $n$ is an odd positive integer then
\begin{align*}
  D_{1,1}(n) &=
  \sum_{k=0}^{(n+1)/2} \!\!\! \rho_k{\left(4(\mu-2),\frac{1}{(2k-1)!}\right)}
    \frac{\bigl(\mu-1\bigr)_{\!3k-2}}{2 \left(\frac{\mu}{2}+k-\frac{1}{2}\right)_{\!k-1}}
  \left(\prod_{j=1}^{k-1}
    \frac{\bigl(\mu+2j+1\bigr)_{\!j-1} \left(\frac{\mu}{2}+2j+\frac{1}{2}\right)_{\!j-1}}
         {\bigl(j\bigr)_{\!j-1} \left(\frac{\mu}{2}+j+\frac{1}{2}\right)_{\!j-1}}
  \right)^{\!\!2}
  \\ & \qquad\qquad\times
  \left(\prod_{j=k}^{(n-1)/2}
    \frac{
      \bigl(\mu+2j\bigr)^{\!2}_{\!j}
      \left(\frac{\mu}{2}+2j-\frac{1}{2}\right)_{\!j}
      \left(\frac{\mu}{2}+2j+\frac{3}{2}\right)_{\!j+1}
    }{
      \bigl(j\bigr)_{\!j} \, \bigl(j+1\bigr)_{\!j+1}
      \left(\frac{\mu}{2}+j+\frac{1}{2}\right)^{\!2}_{\!j}
    }\right).
\end{align*}
If $n$ is an even positive integer then
\begin{align*}
  D_{1,1}(n) &=
  \sum_{k=0}^{n/2} \rho_k{\left(4(\mu-2),\frac{1}{(2k-1)!}\right)}
    \frac{\bigl(\mu-1\bigr)_{\!3k-2}}{2 \left(\frac{\mu}{2}+k-\frac{1}{2}\right)_{\!k-1}}
  \left(\prod_{j=1}^{k-1}
    \frac{\bigl(\mu+2j+1\bigr)_{\!j-1} \left(\frac{\mu}{2}+2j+\frac{1}{2}\right)_{\!j-1}}
         {\bigl(j\bigr)_{\!j-1} \left(\frac{\mu}{2}+j+\frac{1}{2}\right)_{\!j-1}}
  \right)^{\!\!2}
  \\ & \qquad\quad \times
  \left(\prod_{j=k}^{n/2}
    \frac{
      \bigl(\mu+2j\bigr)_{\!j}
      \left(\frac{\mu}{2}+2j+\frac{1}{2}\right)_{\!j-1}
    }{
      \bigl(j\bigr)_{\!j} \,
      \left(\frac{\mu}{2}+j+\frac{1}{2}\right)_{\!j-1}
    }\right)
  \left(\prod_{j=k}^{n/2-1}
    \frac{
      \bigl(\mu+2j\bigr)_{\!j}
      \left(\frac{\mu}{2}+2j+\frac{3}{2}\right)_{\!j+1}
    }{
      \bigl(j+1\bigr)_{\!j+1}
      \left(\frac{\mu}{2}+j+\frac{1}{2}\right)_{\!j}
    }\right).
\end{align*}
\end{thm}
\begin{proof}
Starting from formula~\eqref{eq.cf1} we want to derive the asserted evaluation
of the determinant~$D_{1,1}(n)$.  By noting that $R_{0,0}(0)=1$ we can write
$\prod_{j=1}^n R_{0,0}(j) = \prod_{j=0}^n R_{0,0}(j)$, which allows us to
include it as a first summand into the sum, with some little adaption: the sum
is multiplied by the factor $(\mu-1)/2$, which is missing in the first
term. Moreover, when we want to set $k=0$ in the expression given in
Lemma~\ref{lem.PR1001}, the factorial $(2k-1)!$ in the denominator is
disturbing. Last but not least, when we multiply this expression by $(2k-1)!$
and then set $k=0$, we get $1/\bigl(2(\mu-1)(\mu-2)\bigr)$, and not~$1$.  All
these cases are taken care of by introducing the following $\rho_k$ term:
\[
  \rho_k{\left(2(\mu-1)(\mu-2),\frac{\mu-1}{2(2k-1)!}\right)} =
  \frac{\mu-1}{2} \cdot \rho_k{\left(4(\mu-2),\frac{1}{(2k-1)!}\right)}.
\]
By writing
\[
  \prod_{j=2k}^n R_{0,0}(j) = \left(\prod_{j=k}^{\lfloor n/2\rfloor} R_{0,0}(2j)\right) \left(\prod_{j=k}^{\lfloor (n-1)/2\rfloor} R_{0,0}(2j+1)\right)
\]
we can apply Proposition~\ref{prop.D00}. After putting everything together, and after
some minor simplifications, we obtain the formulas stated in the theorem.
\end{proof}

This derivation not only yields a new, and relatively nice, formula for
$D_{1,1}(n)$, but also explains the emergence of the ``ugly'' polynomial
factor.

\section{Proof of the Monstrous Conjecture}\label{sec.monst}

This section deals with the proof of our own conjecture concerning~$D_{1,1}(n)$.
In a previous paper~\cite{KoutschanThanatipanonda13}, we conjectured that for
every positive integer~$n$ we have
\[
  D_{1,1}(n) = \det_{1\leq i,j\leq n}\left(\delta_{i,j} + \binom{\mu+i+j-2}{j}\right) = 
  \textstyle C(n) F(n) G\!\left(\left\lfloor\frac12(n+1)\right\rfloor\right)
\]
where the quantities $C(n)$, $F(n)$, and $G(n)$ are defined as follows
\begin{align*}
C(n) & = \frac{(-1)^n+3}{2}\prod_{i=1}^n\frac{\left\lfloor\frac{i}{2}\right\rfloor !}{i!},\\
F(n) & = \begin{cases}  
    E(n) F_0(n), & \text{if}\ n\ \text{is even},\\
    \displaystyle E(n) F_1(n) \prod_{i=1}^{(n-5)/2} \!\! (\mu+2i+2n-1), & \text{if}\ n\ \text{is odd},
  \end{cases}\\
E(n) & = (\mu+1)_n 
  \Biggl(\prod_{i=1}^{\lfloor 3/2\lfloor(n-1)/2\rfloor-2\rfloor} \!\!
    \bigl(\mu+2i+6\bigr)^{2\lfloor(i+2)/3\rfloor}\Biggr)\\
  & \qquad \times
  \Biggl(\prod_{i=1}^{\lfloor 3/2\lfloor n/2\rfloor-2\rfloor} \!\!
    \Big(\mu+2i+{\textstyle 2\left\lfloor\frac32\left\lfloor\frac{n}{2}+1\right\rfloor\right\rfloor}-1
    \Big)^{2\lfloor \lfloor n/2\rfloor/2-(i-1)/3\rfloor-1}\Biggr),\\
F_m(n) & = \Biggl(\prod_{i=1}^{\lfloor(n-1)/4\rfloor} \!\!\! \bigl(\mu+2i+n+m\bigr)^{1-2i-m}\Biggr)
  \Biggl(\prod_{i=1}^{\lfloor n/4-1\rfloor} \! \bigl(\mu-2i+2n-2m+1\bigr)^{1-2i-m}\Biggr),\\
G(n) & = \begin{cases}
  P_1\!\left(\frac12(n+1)\right), & \text{if}\ n\ \text{is odd},\\
  P_2\!\left(\frac{n}{2}\right), & \text{if}\ n\ \text{is even}.
  \end{cases}
\end{align*}

$P_1(n)$ and $P_2(n)$ are polynomials in $\mu$, whose definition is quite
involved and therefore not reproduced here. However, it is important to note
that they satisfy, respectively, second-order recurrence relations. Actually,
they were originally found as solutions of these guessed recurrences.

In order to prove our conjecture, we investigate the expression
$D_{1,1}(n)/\bigl(C(n) F(n)\bigr)$, so that the goal is to show that this
expression equals $G\!\left(\left\lfloor\frac12(n+1)\right\rfloor\right)$ for
any positive integer~$n$.  For this purpose, we rewrite the single-sum
expression for $D_{1,1}(n)$ given in Theorem~\ref{thm.D11} by splitting some
of the Pochhammer symbols, so that they either produce factors of the form
$(\mu+2i)$ or $(\mu+2i-1)$, at the cost of introducing some floor functions.
For example, for even~$n$ we obtain:
\begin{multline*}
  \sum_{k=0}^{n/2} \rho_k{\left(\mu-2,\frac{1}{(2k-1)!}\right)}
    \frac{2^{n^2/4-k(k+1)+1} \bigl(\mu-1\bigr)_{\!3k-2}}{\left(\frac{\mu}{2}+k-\frac{1}{2}\right)_{\!k-1}}
  \left(\prod_{j=1}^{k-1}
    \frac{\left(\frac{\mu}{2}+2j+\frac{1}{2}\right)_{\!j-1} \left(\frac{\mu}{2}+j+1\right)_{\!j-2}}
         {\left(\frac{1}{2}\right)_{\!j-1} \bigl(\mu+3j\bigr)_{\!j-2}}
  \right)^{\!\!2}
  \\ \times
  \left(\prod_{j=k}^{n/2}
    \frac{
      \left(\frac{\mu}{2}+j\right)_{\!\lfloor (j+1)/2\rfloor}
      \left(\frac{\mu}{2}+2j+\frac{1}{2}\right)_{\!j-1}
    }{
      \bigl(j\bigr)_{\!j} \,
      \left(\frac{\mu}{2}+\left\lfloor\frac{3j}{2}\right\rfloor+\frac{1}{2}\right)_{\!\lfloor(j-1)/2\rfloor }
    }\right)
  \left(\prod_{j=k}^{n/2-1}
    \frac{
      \left(\frac{\mu}{2}+j\right)_{\!\lfloor(j+1)/2\rfloor }
      \left(\frac{\mu}{2}+2j+\frac{3}{2}\right)_{\!j+1}
    }{
      \bigl(j+1\bigr)_{\!j+1}
      \left(\frac{\mu}{2}+\left\lfloor\frac{3j}{2}\right\rfloor+\frac{1}{2}\right)_{\!\lfloor(j+1)/2\rfloor }
    }\right).
\end{multline*}
Next, we replace all products of the form $\prod_{j=k}^{n/2-c}f(j)$ by the
quotient $\bigl(\prod_{j=1}^{n/2-c}f(j)\bigr)/\bigl(\prod_{j=1}^{k-1}f(j)\bigr)$
(plus some correction for the case $k=0$). Then we can move those factors that
do not depend on~$k$ outside the summation sign. In order to handle the floor
functions, we make a case distinction according to the residue class of $n$
modulo~$4$.  We start by inspecting the case that $n$ is divisible by~$4$,
i.e., $n=4\ell$, $\ell\in\N$; then we have
\begin{align*}
 \frac{D_{1,1}(n)}{C(n) F(n)} &=
 \underbrace{
   \frac{
     \left(\prod_{j=1}^{2\ell-1} \left(\frac{\mu}{2}+j\right)_{\!\lfloor (j+1)/2\rfloor}\right) 
     \left(\prod_{j=1}^{2\ell} \left(\frac{\mu}{2}+j\right)_{\!\lfloor (j+1)/2\rfloor}\right)
   }{
     \left(\prod_{i=1}^{3\ell-4} (\mu+2i+6)^{2\lfloor(i+2)/3\rfloor}\right)
     \left(\prod_{i=1}^{\ell-1} (\mu+4\ell+2i)^{1-2i}\right)
   }
 }_{
   \textstyle = 2^{-(\ell-1)(2\ell-1)} \left(\frac{\mu}{2}+1\right)_{\!2\ell-1} \left(\frac{\mu}{2}+1\right)_{\!2\ell}
 }
 \\[2ex]
 & \qquad\times\quad
 \underbrace{
   \frac{\displaystyle
     \Biggl(\prod_{j=1}^{2\ell}
     \frac{\left(\frac{\mu}{2}+2j+\frac{1}{2}\right)_{\!j-1}}
          {\left(\frac{\mu}{2}+\left\lfloor \frac{3j}{2}\right\rfloor+\frac{1}{2} \right)_{\!\lfloor(j-1)/2\rfloor}}
     \Biggr)
     \Biggl(\prod_{j=1}^{2\ell-1}
     \frac{\left(\frac{\mu}{2}+2j+\frac{3}{2}\right)_{\!j+1}}
          {\left(\frac{\mu}{2}+\left\lfloor \frac{3j}{2}\right\rfloor+\frac{1}{2} \right)_{\!\lfloor (j+1)/2\rfloor}}
     \Biggr)
   }{
     \left(\prod_{i=1}^{3\ell-2} \left(\mu+6\ell+2i+1\right)^{2\ell\lfloor (1-i)/3\rfloor+1}\right)
     \left(\prod_{i=1}^{\ell-1} (\mu+8\ell-2i+1)^{1-2i}\right)
   }
 }_{
   \textstyle = \frac{1}{\mu+3} 2^{-2\ell(\ell-1)} \left(\frac{\mu}{2}+3\ell+\frac{1}{2}\right)_{\!3\ell-1}
 }
 \\[1ex]
 & \qquad\times\quad
 \underbrace{
   \frac{
     1
   }{
     \left(\prod_{j=1}^{2\ell} \bigl(j\bigr)_j\right)
     \left(\prod_{j=1}^{2\ell-1} \bigl(1+j\bigr)_{1+j}\right)
     \left(\prod_{i=1}^{4\ell} \frac{\left\lfloor i/2\right\rfloor !}{i!}\right)
   }
 }_{
   \textstyle = 2^{2\ell}
 }
 \quad\cdot\quad
 \underbrace{
   \frac{1}{2 \, \bigl(1+\mu\bigr)_{\!4\ell} \vphantom{\left(\prod_{i=1}^{4\ell} \frac{\left\lfloor i/2\right\rfloor !}{i!}\right)}}
 }_{
   \textstyle\kern-30pt
   = 2^{-4\ell-1} \left(\frac{\mu}{2}+\frac{1}{2}\right)^{\!-1}_{\!2\ell} \left(\frac{\mu}{2}+1\right)^{\!-1}_{\!2\ell}
   \kern-70pt
 }
 \quad\cdot\quad
 \sum_{k=0}^{2\ell}\bigl(\dots\bigr)
 \\[2ex]
 &= 2^{-4\ell^2+3\ell-2} \,
   \frac{\left(\frac{\mu}{2}+1\right)_{\!2\ell-1} \left(\frac{\mu}{2}+3\ell+\frac{1}{2}\right)_{\!3\ell-1}}
        {(\mu+3) \left(\frac{\mu}{2}+\frac{1}{2}\right)_{\!2\ell}} \,
   \sum_{k=0}^{2\ell}\bigl(\dots\bigr)
\end{align*}

Next, we treat the expression inside the sum, which was abbreviated by
$\bigl(\dots\bigr)$ in the previous calculation. Again, we separate ``even''
and ``odd'' factors by
\[
  \textstyle\bigl(\mu+3j\bigr)_{\!j-2} = 2^{j-2}
  \left(\frac{\mu}{2} + \bigl\lfloor \frac{3j}{2}+\frac{1}{2} \bigr\rfloor\right)_{\!\lfloor (j-2)/2\rfloor}
  \left(\frac{\mu}{2} + \bigl\lfloor \frac{3j}{2} \bigr\rfloor +\frac{1}{2}\right)_{\!\lfloor (j-1)/2\rfloor}.
\]
Then we can simplify as follows:
\begin{align*}
 \bigl(\dots\bigr) &=
   2^{4\ell^2-k^2-k+1} \, \rho_k{\left(\frac{1}{2} (\mu-2) (\mu+3),\frac{1}{(2k-1)!}\right)} \,
   \frac{\bigl(\mu-1\bigr)_{\!3k-2}}{\left(\frac{\mu}{2}+k-\frac{1}{2}\right)_{\!k-1}}
 \\[2ex]
 & \qquad\times\quad
 \underbrace{
   \prod_{j=1}^{k-1} 2^{4-2j}
 }_{
   \textstyle \kern-30pt = 2^{-4+5k-k^2} \rho_k(16,1) \kern-30pt
 }
 \quad\cdot\quad
 \underbrace{
   \prod_{j=1}^{k-1} \frac{\left(\frac{\mu}{2}+j+1\right)^{\!2}_{\!j-2}}
        {\left(\frac{\mu}{2}+j\right)^{\!2}_{\!\lfloor(j+1)/2\rfloor}
         \left(\frac{\mu}{2}+\left\lfloor\frac{3j}{2}+\frac{1}{2}\right\rfloor\right)^{\!2}_{\!\lfloor(j-2)/2\rfloor}}
 }_{
   \textstyle = \rho_k{\left(4\mu^{-2},1\right)}\left(\frac{\mu}{2}+1\right)^{\!-2}_{\!k-1}
 }
 \\[2ex]
 & \qquad\times\quad
 \underbrace{
   \prod_{j=1}^{k-1}
   \frac{
     \left(\frac{\mu}{2}+\left\lfloor\frac{3j}{2}\right\rfloor+\frac{1}{2}\right)_{\!\lfloor(j+1)/2\rfloor}
     \left(\frac{\mu}{2}+2j+\frac{1}{2}\right)_{\!j-1}
   }{
     \left(\frac{\mu}{2}+2j+\frac{3}{2}\right)_{\!j+1}
     \left(\frac{\mu}{2}+\left\lfloor\frac{3j}{2}\right\rfloor+\frac{1}{2}\right)_{\!\lfloor(j-1)/2\rfloor}
   }
 }_{
   \textstyle = \rho_k{\left(1,\frac{\mu+3}{2}\right)} \left(\frac{\mu}{2}+2k-\frac{1}{2}\right)^{\!-1}_{\!k}
 }
 \quad\cdot\quad
 \underbrace{
   \prod _{j=1}^{k-1} \frac{\bigl(j\bigr)_{\!j} \, \bigl(j+1\bigr)_{\!j+1}}{\left(\frac{1}{2}\right)^{\!2}_{\!j-1}}
 }_{
   \textstyle \kern-20pt
   = \rho_k{\left(\frac{1}{8},1\right)} 2^{2 k (k-1)} \left(\frac{3}{2}\right)_{\!k-1} \left(\frac{1}{2}\right)^{\!2}_{\!k-1}
   \kern-60pt
 }
 \\[2ex]
 &= 2^{4\ell^2+2k-3} \, \rho_k{\left(\frac{4(\mu-2)}{\mu^2},\frac{1}{2(2k-1)!}\right)} \,
 \frac{
   (\mu+3)\left(\frac{1}{2}\right)^{\!2}_{\!k-1}
   \left(\frac{3}{2}\right)_{\!k-1} \,
   \bigl(\mu-1\bigr)_{\!3k-2}
 }{
   \left(\frac{\mu}{2}+1\right)^{\!2}_{\!k-1}
   \left(\frac{\mu}{2}+k-\frac{1}{2}\right)_{\!k-1}
   \left(\frac{\mu}{2}+2k-\frac{1}{2}\right)_{\!k}
 }
\end{align*}

Putting everything together yields the following expression for $D_{1,1}(4\ell)/\bigl(C(4\ell) F(4\ell)\bigr)$:
\begin{equation} \label{eq.quot4}
  \sum_{k=0}^{2\ell} 2^{3\ell+k-2} \, \rho_k{\left(\frac{\mu-2}{\mu^2},\frac{1}{8(2k-2)!!}\right)} \,
  \frac{
    \bigl(\mu-1\bigr)_{\!3k-2}
    \left(\frac{1}{2}\right)^{\!2}_{\!k-1}
    \left(\frac{\mu}{2}+k\right)_{\!2\ell-k}
    \left(\frac{\mu}{2}+3\ell+\frac{1}{2}\right)_{3\ell-1}
  }{
    \left(\frac{\mu}{2}+\frac{1}{2}\right)_{\!2\ell}
    \left(\frac{\mu}{2}+1\right)_{\!k-1}
    \left(\frac{\mu}{2}+k-\frac{1}{2}\right)_{\!k-1} 
    \left(\frac{\mu}{2}+2k-\frac{1}{2}\right)_{\!k}
  }.
\end{equation}

We now have to show that \eqref{eq.quot4} equals
$G{\left(\left\lfloor\frac12(n+1)\right\rfloor\right)} = G(2\ell) =
P_2(\ell)$.  We do this by showing that \eqref{eq.quot4} satisfies the same
recurrence as $P_2(\ell)$.  Since we have the case distinction at $k=0$ given
by $\rho_k$, we split the sum as follows:
\[
  \sum_{k=0}^{2\ell} f(\ell,k) = \sum_{k=1}^{2\ell} f(\ell,k) + f(\ell,0),
\]
with
\[
  f(\ell,k) =
  \frac{
    2^{3 \ell+k-5} \bigl(\mu-1\bigr)_{\!3k-2} \left(\frac{1}{2}\right)^{\!2}_{\!k-1}
    \left(\frac{\mu}{2}+k\right)_{\!2\ell-k} \left(\frac{\mu}{2}+3\ell+\frac{1}{2}\right)_{\!3\ell-1}
  }{
    (2k-2)!! \left(\frac{\mu}{2}+\frac{1}{2}\right)_{\!2\ell}
    \left(\frac{\mu}{2}+1\right)_{\!k-1} \left(\frac{\mu}{2}+k-\frac{1}{2}\right)_{\!k-1}
    \left(\frac{\mu}{2}+2k-\frac{1}{2}\right)_{\!k}
  }.
\]
Next, we note that $f(\ell,k)$ satisfies the first-order recurrence
$p_1(\ell) f(\ell+1,k) + p_0(\ell) f(\ell,k)$ with
\begin{align*}
  p_1(\ell) &= (\mu +4 \ell+1) (\mu +4 \ell+3) (\mu +6 \ell+1) (\mu +6 \ell+3) (\mu +6 \ell+5) \\
  p_0(\ell) &= -(\mu +4 \ell) (\mu +4 \ell+2) (\mu +12 \ell-1) (\mu +12 \ell+1) (\mu +12 \ell+3) \\
  &\qquad \times (\mu +12 \ell+5) (\mu +12 \ell+7) (\mu +12 \ell+9)
\end{align*}
whose coefficients $p_0(\ell)$ and $p_1(\ell)$ are both free of~$k$.
Employing operator notation, where $S_\ell$ denotes the shift operator
w.r.t.~$\ell$ and $\bullet$ denotes operator application, we can write:
\begin{align*}
  0 &= \sum_{k=1}^{2\ell} \bigl(p_1(\ell)S_{\ell}+p_0(\ell)\bigr)\bullet f(\ell,k) \\
  &= \bigl(p_1(\ell)S_{\ell}+p_0(\ell)\bigr)\bullet\sum_{k=1}^{2\ell} f(\ell,k)
  - p_1(\ell) \bigl(f(\ell+1,2\ell+1) + f(\ell+1,2\ell+2)\bigr).
\end{align*}
Note that $f(\ell+1,2\ell+1) + f(\ell+1,2\ell+2)$ is a hypergeometric term,
and hence satisfies a first-order recurrence.  In other words, it is
annihilated by some operator of the form $q_1(\ell)S_{\ell}+q_0(\ell)$.  By an
explicit computation, we find
\begin{align*}
 q_1(\ell) &= (\ell+1) (2 \ell+3) (\mu+4 \ell+4) (\mu+4 \ell+6) (\mu+8 \ell+3) 
 \bigl(2 \mu^5 \ell+\mu^5+152 \mu^4 \ell^2+ \dots
 +420\bigr), \\
 q_0(\ell) &= -8 (4 \ell+1)^2 (4 \ell+3)^2 (\mu+6 \ell) (\mu+6 \ell+2) (\mu+6 \ell+4)
 (\mu+8 \ell+11) \bigl(2 \mu^5 \ell+ \dots
 +797916\bigr),
\end{align*}
where the dots hide, for the convenience of the reader, two irreducible
polynomials that are unhandy to display (each of them is several lines long).

It follows that $\sum_{k=1}^{2\ell} f(\ell,k)$ is annihilated by the product of the two operators
\begin{align*}
  A &= \bigl(q_1(\ell)S_{\ell}+q_0(\ell)\bigr)\cdot\bigl(p_1(\ell)S_{\ell}+p_0(\ell)\bigr) \\
  &= p_1(\ell+1)q_1(\ell)S_{\ell}^2 + \bigl(p_0(\ell+1)q_1(\ell) + p_1(\ell)q_0(\ell)\bigr)S_{\ell} + p_0(\ell)q_0(\ell).
\end{align*}

By a quick computer calculation, we can verify that this operator~$A$ also
annihilates $f(\ell,0)$, namely that the first-order operator killing
$f(\ell,0)$ is a right factor of~$A$, and hence $A$ annihilates also the sum
$\sum_{k=0}^{2\ell} f(\ell,k)$. We compare the operator~$A$ with the operator
that we guessed previously and whose solution yielded the family of
polynomials $P_2(\ell)$.  We find that both operators are identical. A routine
calculation confirms that \eqref{eq.quot4} equals $P_2(\ell)$ for $\ell=1$ and
$\ell=2$. This completes the proof, for the case $n=4\ell$, that the
conjectured formula in~\cite{KoutschanThanatipanonda13} agrees with the (much
simpler) formula that we derived in Section~\ref{sec.nice}.

We have to continue and treat the cases $n=4\ell-1$, $n=4\ell-2$, and $n=4\ell-3$ individually.
They can be done analogously, and we spare the reader from the details of the calculations,
which can be found in~\cite{EM}. To conclude, let
\[
  h(n,k) = 2^k \rho_k{\left(\frac{\mu-2}{\mu^2},\frac{1}{8(2k-2)!!}\right)}
  \frac{
    \left(\frac{1}{2}\right)^{\!2}_{\!k-1} \, \bigl(\mu-1\bigr)_{\!3k-2}
  }{
    \left(\frac{\mu}{2}+1\right)_{\!k-1}
    \left(\frac{\mu}{2}+k-\frac{1}{2}\right)_{\!k-1}
    \left(\frac{\mu}{2}+2k-\frac{1}{2}\right)_{\!k}
  }
\]
(this is the common factor that appears in all four cases). Using this notation,
we obtain the following result:
\[
  \frac{D_{1,1}(n)}{C(n) F(n)} =
  \begin{cases}
  \rule[-20pt]{0pt}{20pt}\displaystyle\sum_{k=0}^{n/2} 2^{(3n-8)/4} h(n,k)
  \frac{\left(\frac{\mu}{2}+k\right)_{\!n/2-k} \left(\frac{\mu}{2}+\frac{3n}{4}+\frac{1}{2}\right)_{\!(3n-4)/4}}
       {\left(\frac{\mu}{2}+\frac{1}{2}\right)_{\!n/2}}, &
  n\equiv0\!\mod4
  \\
  \rule[-20pt]{0pt}{20pt}\displaystyle\sum_{k=0}^{n/2} 2^{(3n-6)/4} h(n,k)
  \frac{\left(\frac{\mu}{2}+k\right)_{\!n/2-k} \left(\frac{\mu}{2}+\frac{3n}{4}\right)_{\!(3n-2)/4}}
       {\left(\frac{\mu}{2}+\frac{1}{2}\right)_{\!n/2}}, &
  n\equiv2\!\mod4
  \\
  \rule[-20pt]{0pt}{20pt}\displaystyle\sum_{k=0}^{(n+1)/2} 2^{(3n-3)/4} h(n,k)
  \frac{\left(\frac{\mu}{2}+k\right)_{\!(n+1)/2-k} \left(\frac{\mu}{2}+\frac{3n}{4}+\frac{3}{4}\right)_{\!(3n+1)/4}}
       {\left(\frac{\mu}{2}+\frac{1}{2}\right)_{\!(n+1)/2}}, &
  n\equiv1\!\mod4
  \\
  \displaystyle\sum_{k=0}^{(n+1)/2} 2^{(3n-5)/4} h(n,k)
  \frac{\left(\frac{\mu}{2}+k\right)_{\!(n+1)/2-k} \left(\frac{\mu}{2}+\frac{3n}{4}+\frac{5}{4}\right)_{\!(3n-1)/4}}
       {\left(\frac{\mu}{2}+\frac{1}{2}\right)_{\!(n+1)/2}}, &
  n\equiv3\!\mod4
  \end{cases}
\]

The above equations can be viewed as an alternative closed form for
$D_{1,1}(n)$.  In particular, they give nicer formulas for the ``ugly''
polynomials $P_1(n)$ and $P_2(n)$, compared to the ones presented
in~\cite{KoutschanThanatipanonda13}.

\section{The General Determinant}\label{sec.gen}

We now want to study the general determinant~$D_{s,t}(n)$, of which the
results in Section~\ref{sec.lemmas} were just special cases.  Indeed, once
several instances of $D_{s,t}(n)$ are settled, it is a natural question to ask
what happens for other values of~$s$ and~$t$. Unfortunately, it seems that
there is no nice formula for general~$s$ and~$t$, but at least we can identify
some infinite families of determinants that give nice evaluations.  Before
stating our results, we give a schematic overview.  We classify several
infinite families of determinants of the form $D_{s,t}(n)$ according to their
factorization properties. Notice that not all of them are proved. In this
context, a polynomial (or rational function) is called ``nice'' if it factors
completely.
\begin{center}
\begin{tabular}{c|ll}
Family & Property & Reference \\ \hline
 0   & $D_{s,t}(n)=0$ & Proposition~\ref{prop.fam0} \rule{0pt}{12pt} \\
 A   & $D_{s,t}(n)$ is nice & Theorem~\ref{thm.famA} \\
A$'$ & $D_{s,t}(n)$ is nice & Corollary~\ref{cor.Dmrmr} \\ 
 B   & $D_{s,t}(2n-1)$ is nice, $D_{s,t}(2n)=0$ & Theorem~\ref{thm.famB} \\
 C   & $D_{s,t}(2n)$ is nice & Conjecture~\ref{conj.famC} \\
 D & $D_{s,t}(2n)$ is nice & Conjecture~\ref{conj.famD} \\
 E   & $D_{s,t}(2n)/D_{s,t}(2n-1)$ is nice & Corollary~\ref{cor.famE} \\
 F   & $D_{s,t}(2n+1)/D_{s,t}(2n)$ is nice & Corollary~\ref{cor.famF}
\end{tabular}
\end{center}

The distribution of these families in the $s$-$t$-plane is shown below; bold
entries mark cases that have been treated in Sections~\ref{sec.lemmas}
and~\ref{sec.nice}. The empty places correspond to choices for $(s,t)$ for
which neither $D_{s,t}(n)$ nor any of its successive quotients is nice.
\begin{center}
{
\setlength{\tabcolsep}{0pt}
\begin{tabular}{c|cccccccccccc}
$t\setminus s\ $ &
$\ \cdots$ &
\parbox{16pt}{\centerline{$-3$}} &
\parbox{16pt}{\centerline{$-2$}} &
\parbox{16pt}{\centerline{$-1$}} &
\parbox{16pt}{\centerline{$0$}} &
\parbox{16pt}{\centerline{$1$}} &
\parbox{16pt}{\centerline{$2$}} &
\parbox{16pt}{\centerline{$3$}} &
\parbox{16pt}{\centerline{$4$}} &
\parbox{16pt}{\centerline{$5$}} &
\parbox{16pt}{\centerline{$6$}} &
$\cdots$
\\ \hline
$\vdots$ & & & & $\vdots$ & $\vdots$ & $\vdots$ & & & & & \\
$6$  & &      &      &  D   & A     & C     &       &   &   &   &   & \\
$5$  & &      &      &  F   & B     & E     &       &   &   &   &   & \\
$4$  & &      &      &  D   & A     & C     &       &   &   &   &   & \\
$3$  & &      &      &  F   & B     & E     &       &   &   &   &   & \\
$2$  & &      &      &  D   & \bf A & C     &       &   &   &   &   & \\
$1$  & &      &      &  F   & \bf B & \bf E & C     & E & C & E & C & $\cdots$ \\
$0$  & &      &      &  D   & \bf A & \bf B & \bf A & B & A & B & A & $\cdots$ \\
$-1$ & &      &      & A$'$ & 0     & 0     & 0     & 0 & 0 & 0 & 0 & $\cdots$ \\
$-2$ & &      & A$'$ &  0   & 0     & 0     & 0     & 0 & 0 & 0 & 0 & $\cdots$ \\
$-3$ & & A$'$ &  0   &  0   & 0     & 0     & 0     & 0 & 0 & 0 & 0 & $\cdots$ \\[-1ex]
$\vdots$ & $\iddots$ & $\iddots$ & $\vdots$ & $\vdots$ & $\vdots$ & $\vdots$ &
$\vdots$ & $\vdots$ & $\vdots$ & $\vdots$ & $\vdots$ & $\ddots$
\end{tabular}
}
\end{center}

Since in these families one of the parameters $s,t$ goes to infinity, we
encounter the situation that for small~$n$ the determinant~$D_{s,t}(n)$
reduces to a simple one, namely one where only the binomial coefficient but
not the Kronecker delta is present. This determinant is well-known, but for
sake of completeness we include it here; also its proof is very simple
(compare also \cite[Sec.~2.3]{Krattenthaler99}).
\begin{prop}\label{prop.bindet}
For $n,s,t\in\Z$, $t\geq 0, n\geq1$, and $\mu$ an indeterminate, we have that
\[
  \det_{\genfrac{}{}{0pt}{}{s \leq i < s+n}{t \leq j < t+n}}
  \left( { \binom{\mu+i+j-2}{j}}\right) =
  \prod_{i=0}^{t-1} \dfrac{\left(\mu+s+i-1 \right)_n}{\left(i+1 \right)_n} =: G_{s,t}(n).
\]
\end{prop}
\begin{proof}
We perform induction on $n$, using~\eqref{eq.DJD} (see p.~\pageref{eq.DJD}).
It is routine to check that the statement is true for the base cases $n=1$ and $n=2$,
and that
\[
  G_{s,t}(n) = \frac{G_{s,t}(n-1)G_{s+1,t+1}(n-1)-G_{s+1,t}(n-1)G_{s,t+1}(n-1)}{G_{s+1,t+1}(n-2)}.
\]
\end{proof}

\begin{cor}[Family~A$'$]\label{cor.Dmrmr}
Let $r\geq0$ be an integer, and let $D_{s,t}(n)$ be the determinant defined
in Definition~\ref{def.Dst}. Then the following holds:
\[
  D_{-r,-r}(n) = \begin{cases}
    D_{0,0}(n-r), & \text{if } r < n, \\
    1, & \text{if } r \geq n.
  \end{cases}
\]
\end{cor}
\begin{proof}
For $r=0$ there is nothing to show. For $r>0$ the corresponding matrix has the
first unit vector in its first column. In its lower-right $(n-1)\times(n-1)$
block the entries are the same as in the matrix of $D_{-r+1,-r+1}(n-1)$. Hence
$D_{-r,-r}(n)=D_{-r+1,-r+1}(n-1)$ and by unrolling this recurrence, the
assertion follows.
\end{proof}

\begin{prop}\label{prop.fam0}
$D_{s,t}(n) = 0$ for $t \leq -1$ and $s \geq t+1$.
\end{prop}
\begin{proof}
The first column of the matrix contains only $0$ as $\delta_{i,t}=0$ for all $i \geq t+1$ and 
\[
  \binom{\mu+i+t-2}{t} = 0 \qquad \mbox{ for all } i.
\]
Therefore the determinant is $0$.
\end{proof}  

The following theorem allows us to switch the values of $s$ and~$t$.
Therefore, we will afterwards only concentrate on the cases $s\geq t$.

\begin{thm}\label{thm.Dst2Dts}
For integers $t\geq s\geq0$ and $n\geq1$, and for an indeterminate~$\mu$, we have
\[
  D_{s,t}(n) = \Biggl(\prod_{i=0}^{t-s-1} \frac{\bigl(\mu+i+s-1\bigr)_{\!n}}{\bigl(i+s+1\bigr)_{\!n}}\Biggr) \cdot D_{t,s}(n).
\]
\end{thm}
\begin{proof}
We prove the statement by induction on~$n$. The base cases $n=1$ and $n=2$ can
be checked by a routine calculation.  Obviously the statement is true for
$s=t$. Now assume that $t>s$.  Using our ``all-purpose weapon'' \eqref{eq.DJD}
(see p.~\pageref{eq.DJD}), the induction step can be done in a
straight-forward way:
\begin{align*}
  D_{s,t}(n) &= \frac{D_{s,t}(n-1) D_{s+1,t+1}(n-1) - D_{s+1,t}(n-1) D_{s,t+1}(n-1)}{D_{s+1,t+1}(n-2)} \\
  &= \frac{
    \Bigl(D_{t,s}(n-1) D_{t+1,s+1}(n-1) - D_{t,s+1}(n-1) D_{t+1,s}(n-1)\Bigr)
    \cdot \prod\limits_{i=0}^{t-s-1} 
    \frac{\left(\mu+i+s-1\right)_{\!n-1}\left(\mu+i+s\right)_{\!n-1}}
         {\left(i+s+1\right)_{\!n-1}\left(i+s+2\right)_{\!n-1}}
  }{
    D_{t+1,s+1}(n-2) \cdot \prod\limits_{i=0}^{t-s-1} \frac{\left(\mu+i+s\right)_{\!n-2}}{\left(i+s+2\right)_{\!n-2}}
  } \\
  &= \frac{D_{t,s}(n-1) D_{t+1,s+1}(n-1) - D_{t,s+1}(n-1) D_{t+1,s}(n-1)}{D_{t+1,s+1}(n-2)}
  \cdot \prod_{i=0}^{t-s-1} \frac{\bigl(\mu+i+s-1\bigr)_{\!n}}{\bigl(i+s+1\bigr)_{\!n}},
\end{align*}
which is exactly the asserted right-hand side, by applying \eqref{eq.DJD} in
the opposite direction.
\end{proof}

\begin{thm}[Family~A]\label{thm.famA}
Let $\mu$ be an indeterminate and let $r\geq0$ and $n>2r$ be integers. Then
\[
  D_{2r,0}(n) = 2 \cdot\! \prod_{i=2r+1}^{n-1} \! R_{2r,0}(i),
\]
where
\begin{align*}
  R_{2r,0}(2n) &=
   \frac{
    \bigl(\mu+2n+4r\bigr)_{\!n-r}
    \left(\frac{\mu}{2}+2n+r+\frac{1}{2}\right)_{\!n-r-1}
  }{
    \bigl(n-r\bigr)_{\!n-r}
    \left(\frac{\mu}{2}+n+2r+\frac{1}{2}\right)_{\!n-r-1}
  }, \\[1ex]
  R_{2r,0}(2n-1) &=
 \frac{
    \bigl(\mu+2n+4r-2\bigr)_{\!n-r-1}
    \left(\frac{\mu}{2}+2n+r-\frac{1}{2}\right)_{\!n-r}
  }{
    \bigl(n-r\bigr)_{\!n-r}
    \left(\frac{\mu}{2}+n+2r-\frac{1}{2}\right)_{\!n-r-1}
  }.
\end{align*}
Hence, we have $R_{2r,0}(n) = D_{2r,0}(n+1) / D_{2r,0}(n)$. 
\end{thm}
\begin{proof}
Before we start with the actual proof, we note that $D_{2r,0}(n)=1$ if
$n\leq2r$; this is a consequence of Proposition~\ref{prop.bindet}.  The
value~$1$ can also be explained combinatorially: We note that $t=0$ implies
that there is no boundary line, but the three other triangular holes are
attached directly to the corners of the central triangular hole. Moreover, the
size of these three triangles is given by~$2r$, and if their size is equal
to~$n$, they divide the tiling region into three non-connected lozenges (left
part of Figure~\ref{fig.families}).  Since there is only one way to tile a
lozenge-shaped region with rhombi, we get $D_{2r,0}(n)=1$.

If $n>2r$ and $\mu\geq2$, the situation looks similar to the one displayed in
the right part of Figure~\ref{fig.families} (for the moment, ignore the shaded
regions and the dashed line). By the previous argument, the light-gray shaded
lozenges can be tiled in a unique way, and hence they can be declared to be
holes, without changing the tiling count. This way we obtain a hexagonal
region with a single, big triangular hole. Note that it is exactly the type of
region whose cyclically symmetric rhombus tilings are counted by $D_{0,0}(n)$.

The size of this hole is $\mu-2+6r$, which is just the sum of the sizes of the
four holes. The distance from the hole to the boundary is given by
$n-2r$. Since in Family~A we have that $s-t$ is even, we are counting all
cyclically symmetric rhombus tilings (without negative weights), and hence
\[
  D_{2r,0}(n) = D_{0,0}(n-2r)\big|_{\mu\to\mu+6r}.
\]
Note that this identity actually holds for all~$\mu$, since for fixed~$n$ we
have polynomials in~$\mu$ on both sides, that agree for infinitely many
values.  The proof is completed by noting that the above expressions for
$R_{2r,0}(n)$ follow immediately from those for $R_{0,0}(n)$ in
Proposition~\ref{prop.D00} by replacing $n$ by $n-r$ and $\mu$ by $\mu+6r$.
\end{proof}

Note that Lemma~\ref{lem.R20} now follows as a special case of
Theorem~\ref{thm.famA}.  The closed form for the other members of Family~A,
namely the determinants of the form $D_{0,2r}(n)$, are obtained by combining
Theorems~\ref{thm.famA} and~\ref{thm.Dst2Dts}.

\begin{figure}
\begin{center}
\includegraphics[width=0.4\textwidth]{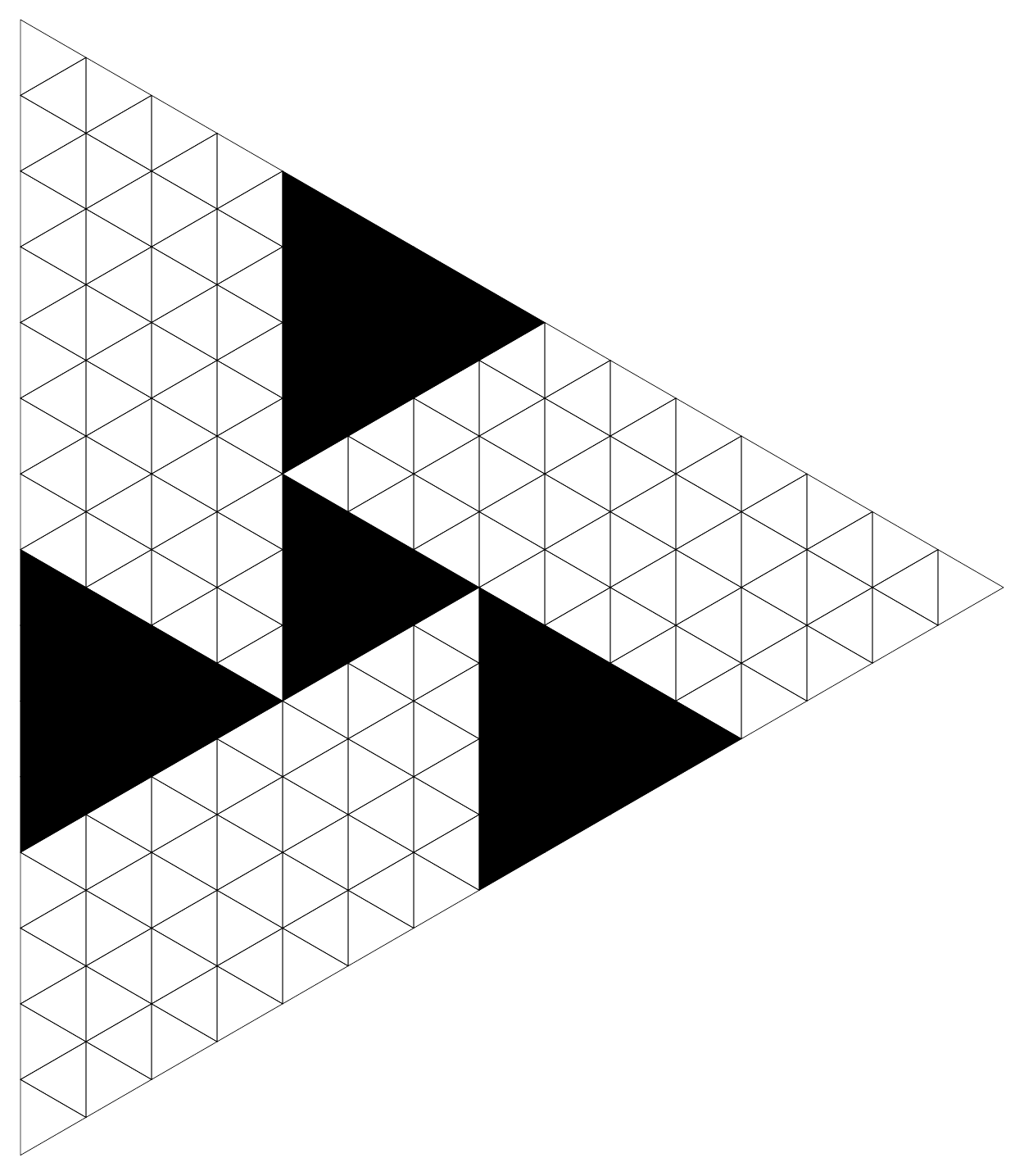}
\qquad\qquad
\includegraphics[width=0.4\textwidth]{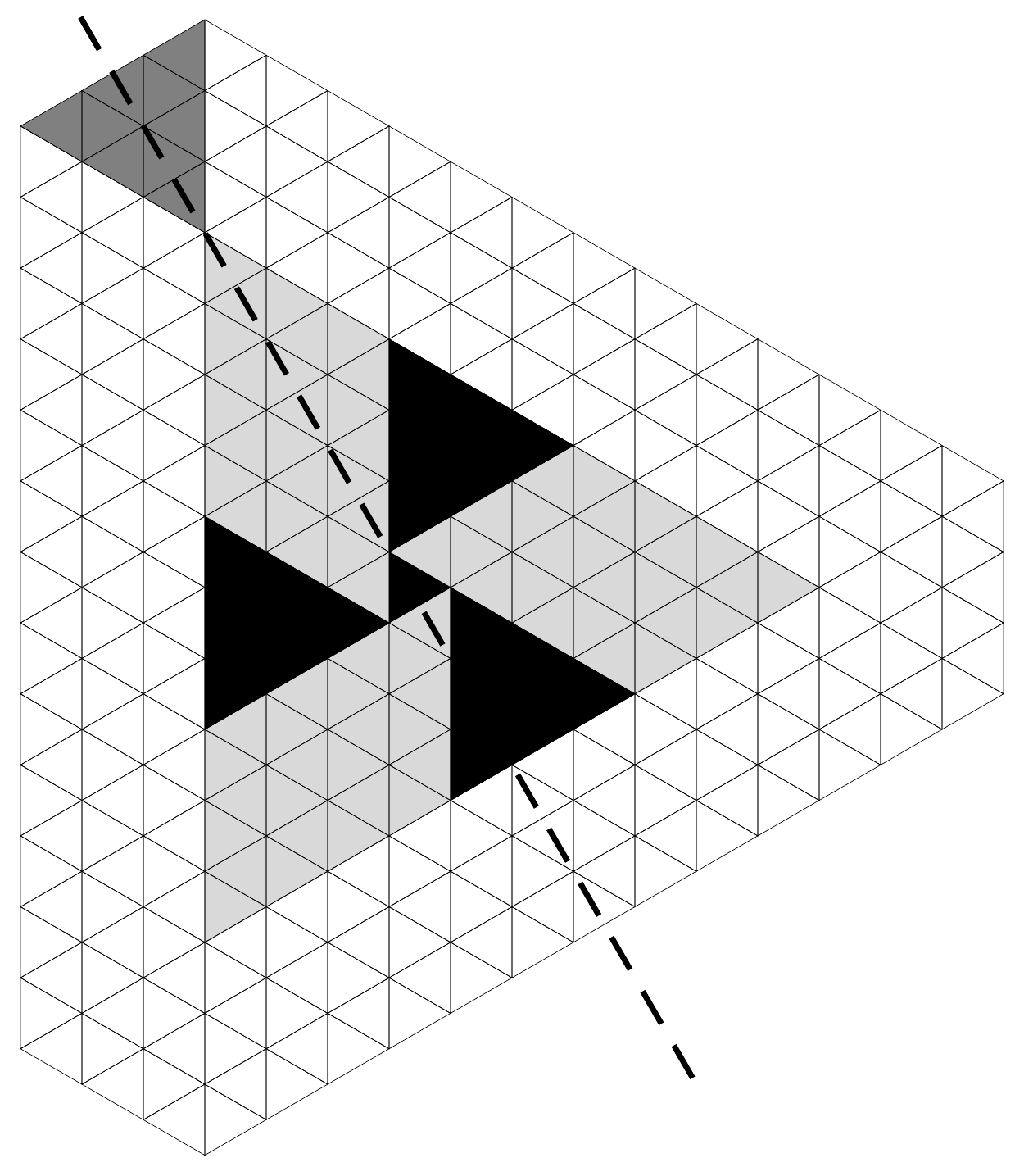}
\end{center}
\caption{Two hexagonal regions with holes, corresponding to Families~A and~B;
on the left with parameters $s=4$, $t=0$, $n=4$, $\mu=5$, on the
right with parameters $s=3$, $t=0$, $n=6$, $\mu=3$.}
\label{fig.families}
\end{figure}

\begin{thm}[Family~B]\label{thm.famB}
Let $\mu$ be an indeterminate, and let $r$ and $n$ be positive integers.  If
$n$ is an odd number, then
\[
  D_{2r-1,0}(n) = \prod_{i=r}^{(n-1)/2} \! R_{2r-1,0}(i),
\]
where
\[
  R_{2r-1,0}(n) = -\frac{
    \bigl(\mu+2n+4r-4\bigr)_{\!n-r+1} \,
    \bigl(\mu+2n+4r-3\bigr)_{\!n-r}
    \left(\frac{\mu}{2}+2n+r-\frac{1}{2}\right)^{\!2}_{\!n-r}
  }{
    \bigl(n-r+1\bigr)_{\!n-r+1} \,
    \bigl(n-r+1\bigr)_{\!n-r}
    \left(\frac{\mu}{2}+n+2r-\frac{3}{2}\right)^{\!2}_{\!n-r}
  }.
\]
Hence, for $n\geq r$ we have $R_{2r-1,0}(n) = D_{2r-1,0}(2n+1) / D_{2r-1,0}(2n-1)$.
If $n\geq 2r$ is an even number, then $D_{2r-1,0}(n)=0$.
\end{thm}
\begin{proof}
According to Proposition~\ref{prop.bindet} we have $D_{2r-1,0}(n)=1$ if
$n<2r$.  When $n$ is odd this is compatible with the asserted formula, since
in this case the product is empty.

The tiling regions corresponding to Family~B look like the ones for Family~A
(with the difference that the three outer holes have odd sizes). We first give
a combinatorial argument for the case when $n\geq 2r$ is even, i.e. the case
where the determinant vanishes.  An example for this situation is displayed on
the right part of Figure~\ref{fig.families}: By declaring the light-gray
lozenges to be holes, we get a hexagonal region with a single triangular hole,
as described before.  The difference now is that $D_{2r-1,0}(n)$ performs a
weighted count. This is the reason for the value~$0$ for even~$n$, since there
are as many tilings with weight~$+1$ as there are with weight~$-1$.  This can
be seen as follows.

In Figure~\ref{fig.families} (right picture) we identify the border of the
original lozenge-shaped region: its left vertical side starts at the top-most
vertex of the smallest black triangle. The lower $2r-1$ unit segments of this
side lie inside the black region, while each of the upper $n-2r+1$ unit
segments may or may not be covered by a rhombus when the whole region is
tiled. For a particular (cyclically symmetric) rhombus tiling, the number of
unit segments which are not crossed by a horizontal rhombus corresponds to the
cardinality of the set~$I$ in~\eqref{eq.SumOfMinors}, and hence its parity
determines whether this tiling is counted with weight $+1$ or with weight $-1$
(note that $s-t=2r-1$ is odd).

We now look at the lozenge-shaped region between the upper part of the
above-mentioned vertical line and the dark-gray shaded triangle (see the right
part of Figure~\ref{fig.families}); the tilings of this lozenge correspond to
a rectangle in which $|I|$ lattice paths connect two opposite sides.  Hence
there are also $|I|$ horizontal rhombi crossing the vertical side of the
dark-gray triangle. Inside the dark-gray triangle a rhombus tiling corresponds
to paths that start at the $|I|$ horizontal rhombi; this situation is depicted
in Figure~\ref{fig.triangle} where the start positions are shown as black
rhombi. Each path must end somewhere on the lower side of the triangle and its
last rhombus will share an edge with the boundary of the triangle. All other
segments of the lower side are crossed by rhombi (also colored black in
Figure~\ref{fig.triangle}). We see that any tiling with $|I|$ rhombi crossing
the vertical side of the triangle forces $n-2r+1-|I|$ rhombi to cross its
other side.

By considering the reflection across the dashed line in
Figure~\ref{fig.families}, one recognizes that there are as many cyclically
symmetric tilings with $|I|$ rhombi crossing the vertical side of the
dark-gray triangle as there are with $n-2r+1-|I|$ such rhombi. Hence, if
$n-2r+1$ is an odd number, the weighted count yields~$0$. Note that this
argument establishes an alternative proof of Lemma~\ref{lem.E10}.

However, to prove the full statement of the theorem, we take a different
approach (which also covers the already discussed cases). Similar to the proof
of Theorem~\ref{thm.famA}, one can reduce $D_{2r-1,0}$ to $D_{1,0}$.
$D_{2r-1,0}(n)$ corresponds to a triangular hole of size $\mu+6r-5$ whose
distance to the boundary of the hexagon is $n-2r+1$, while for $D_{1,0}(n)$ we
have a hole of size $\mu+1$ and distance $n-1$. Hence
\[
  D_{2r-1,0}(n) = D_{1,0}(n-2r+2)\big|_{\mu\to\mu+6r-6}.
\]
The proof is completed by noting that the above expression for $R_{2r-1,0}(n)$
follows immediately from the one for $R_{1,0}(n)$ in
Proposition~\ref{prop.D10} by replacing $n$ by $n-r+1$ and $\mu$ by
$\mu+6r-6$.
\end{proof}

\begin{figure}
  \begin{center}
    \includegraphics[width=0.4\textwidth]{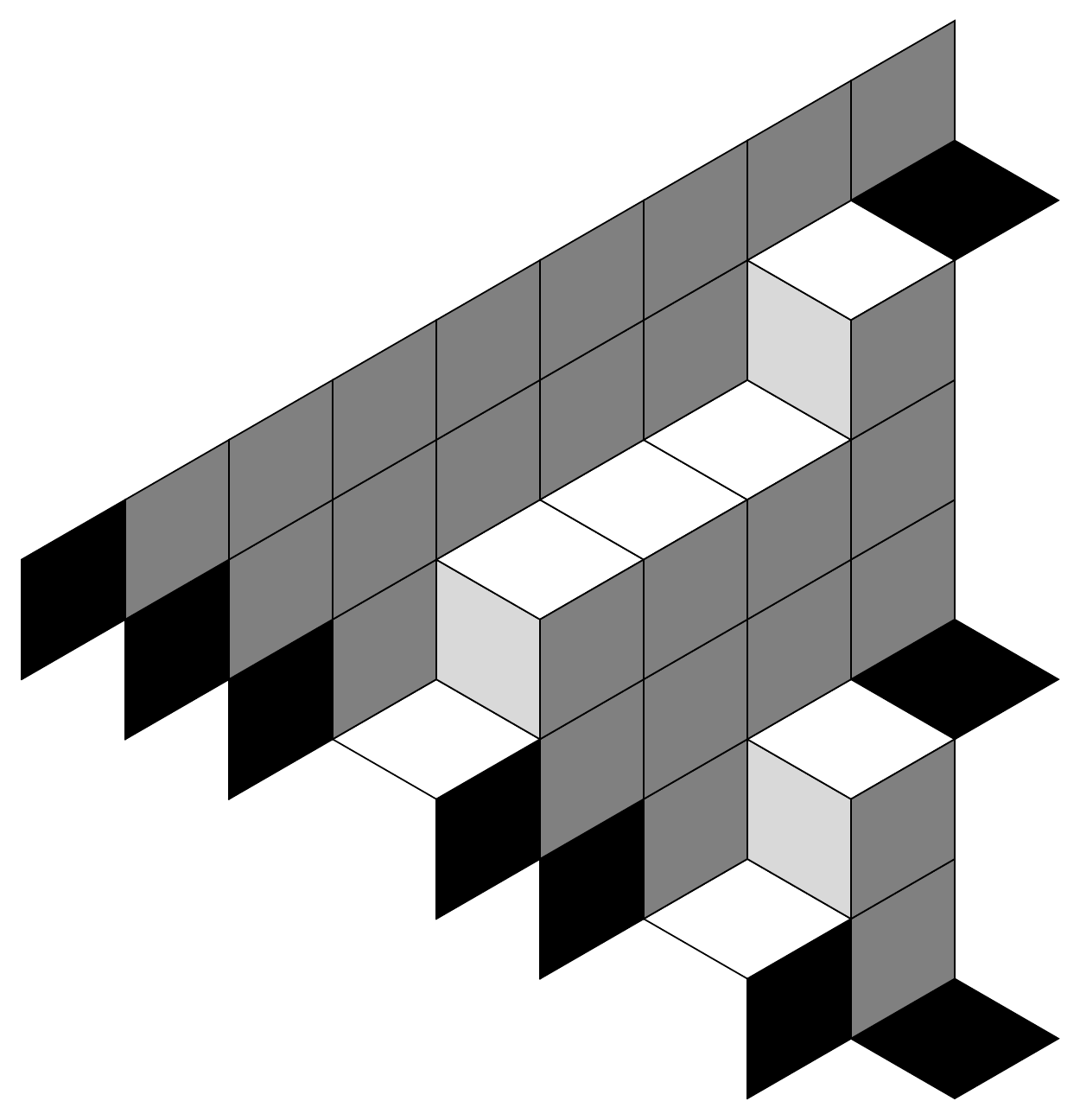}
  \end{center}
  \caption{A tiled triangular region of size~$9$ with $|I|=3$ ``paths''
    entering from the right, of lengths $7$, $3$, and $0$, respectively;
    consequently, $9-3=6$ rhombi have to cross the lower left boundary of the
    region.}
  \label{fig.triangle}
\end{figure}

The tiling regions for Families~C and~D are more complicated, and in particular
we cannot simplify the different holes and borders to a single large triangular
hole. For this reason, the proof strategy applied to Families~A and~B does not
work. So far we have not been able to come up with a proof and therefore we
state the following two formulas as conjectures.

\begin{conj}[Family~C]\label{conj.famC}
Let $\mu$ be an indeterminate and let $n$ and $r$ be positive integers.
If $n\geq r$ then
\[
  D_{2r,1}(2n) = \frac{(\mu-1) \, \bigl(\mu+2r\bigr)_{\!2r-1}}{(2r)!}
  \cdot \prod_{i=r}^{n-1} R_{2r,1}(i),
\]
where
\[
  R_{2r,1}(n) = -\frac{
    (2n+2r) (\mu+2n+2r-1) (\mu+2n+2r) \,
    \bigl(\mu+2n+4r\bigr)^{\!2}_{\!n-r}
    \left(\frac{\mu}{2}+2n+r+\frac{3}{2}\right)^{\!2}_{\!n-r+1}
  }{
    (2n+1)(2n+2)(\mu+2n+1)
    \bigl(n-r+1\bigr)^{\!2}_{\!n-r+1}
    \left(\frac{\mu}{2}+n+2r+\frac{1}{2}\right)^{\!2}_{\!n-r}
  }.
\]
Hence we have that $R_{2r,1}(n) = D_{2r,1}(2n+2) / D_{2r,1}(2n)$.
\end{conj}

Note that for $n<r$ we have $D_{2r,1}(2n) = \bigl(\mu+2r-1\bigr)_{\!2n}/(2n)!$,
according to Proposition~\ref{prop.bindet}. The conjectured closed form for $D_{1,2r}(n)$
can be obtained via Theorem~\ref{thm.Dst2Dts}.

\begin{conj}[Family~D]\label{conj.famD}
Let $\mu$ be an indeterminate and let $n\geq1$ and $r\geq0$ be integers.  Then
\[
  D_{-1,2r}(2n) = \prod_{i=0}^{n-1} R_{-1,2r}(i),
\]
where
\[
  R_{-1,2r}(n) = \begin{cases}
    \displaystyle -\frac{
      \bigl(\mu+2n-1\bigr)_{\!2r} \,
      \bigl(\mu+2n-3\bigr)_{\!2r+1} \,
      \bigl(\mu+2n+4r\bigr)^{\!2}_{\!n-r}
      \left(\frac{\mu}{2}+2n+r+\frac{1}{2}\right)^{\!2}_{\!n-r-1}
    }{
      \bigl(2n+1\bigr)_{\!2r} \,
      \bigl(2n+2\bigr)_{\!2r+1} \,
      \bigl(n-r\bigr)^{\!2}_{\!n-r}
      \left(\frac{\mu}{2}+n+2r+\frac{1}{2}\right)^{\!2}_{\!n-r-1}
    }, & \text{if } n>r, \\
    \rule{0pt}{24pt}\displaystyle
    \frac{
      (3-\mu) \, \bigl(\mu+2r-2\bigr)_{\!2r} \, \bigl(\mu+2r-1\bigr)_{\!2r}
    }{
      \bigl(2r+1\bigr)_{\!2r} \, \bigl(2r+1\bigr)_{\!2r+1}
    }, & \text{if } n=r, \\
    \rule{0pt}{24pt}\displaystyle
    \frac{
      \bigl(\mu+2n-2\bigr)_{\!2r} \, \bigl(\mu+2n-1\bigr)_{\!2r}
    }{
      \bigl(2n+1\bigr)_{\!2r} \, \bigl(2n+2\bigr)_{\!2r}
    }, & \text{if } n<r.
  \end{cases}
\]
Hence, we have that $R_{-1,2r}(n) = D_{-1,2r}(2n+2) / D_{-1,2r}(2n)$.
\end{conj}

\begin{cor}[Family~E]\label{cor.famE}
Let $\mu$ be an indeterminate and let $n \geq r \geq 1$ be integers. Then:
\begin{align*}
  \frac{D_{2r-1,1}(2n)}{D_{2r-1,1}(2n-1)} &=
  \frac{
    \bigl(\mu+2n+4r-4\bigr)_{\!n-r+1}
    \left(\frac{\mu}{2}+2n+r-\frac{1}{2}\right)_{\!n-r}
  }{
    \bigl(n-r+1\bigr)_{\!n-r+1}
    \left(\frac{\mu}{2}+n+2r-\frac{3}{2}\right)_{\!n-r}
  },
  \\[2ex]
  \frac{D_{1,2r-1}(2n)}{D_{1,2r-1}(2n-1)} &=
  \frac{
    \bigl(\mu+2n-1\bigr)_{\!2r-2} \,
    \bigl(\mu+2n+4r-4\bigr)_{\!n-r+1}
    \left(\frac{\mu}{2}+2n+r-\frac{1}{2}\right)_{\!n-r}
  }{
    \bigl(2n+1\bigr)_{\!2r-2} \,
    \bigl(n-r+1\bigr)_{\!n-r+1}
    \left(\frac{\mu }{2}+n+2r-\frac{3}{2}\right)_{\!n-r}
  }.
\end{align*}
\end{cor}
\begin{proof}
Using Theorems~\ref{thm.famA} and~\ref{thm.famB}, we can express the above
quotients in terms of known determinants, by using the Desnanot-Jacobi-Dodgson
identity~\eqref{eq.DJD}:
\vspace{-8pt}
\[
  D_{2r-2,0}(2n+1)D_{2r-1,1}(2n-1) = D_{2r-2,0}(2n)D_{2r-1,1}(2n) 
  -\cancelto{0}{D_{2r-1,0}(2n)}D_{2r-2,1}(2n),
\]
where $D_{2r-1,0}(2n)=0$ only if $n \geq r$. Therefore
\[
  \frac{D_{2r-1,1}(2n)}{D_{2r-1,1}(2n-1)} = \frac{D_{2r-2,0}(2n+1)}{D_{2r-2,0}(2n)}
  \qquad (n\geq r).
\]
The following fact can be derived similarly:    
 \[
  \frac{D_{1,2r-1}(2n)}{D_{1,2r-1}(2n-1)} = \frac{D_{0,2r-2}(2n+1)}{D_{0,2r-2}(2n)}
  \qquad (n\geq r).
\]
\end{proof}

\begin{cor}[Family~F]\label{cor.famF}
Let $\mu$ be an indeterminate and let $n \geq r \geq 1$ be integers. Then:
\[
  \frac{D_{-1,2r-1}(2n+1)}{D_{-1,2r-1}(2n)} =
  \frac{
    2 \, \bigl(\mu+2n-2\bigr)_{\!2r} \,
    \bigl(\mu+2n+4r-2\bigr)_{\!n-r-1}
    \left(\frac{\mu}{2}+2n+r-\frac{1}{2}\right)_{\!n-r}
  }{
    \bigl(2n\bigr)_{\!2r} \,
    \bigl(n-r+1\bigr)_{\!n-r}
    \left(\frac{\mu}{2}+n+2r-\frac{1}{2}\right)_{\!n-r-1}
  }.
\]
\end{cor}
\begin{proof}
Using Theorems~\ref{thm.famA} and~\ref{thm.famB}, we can express the quotient
in terms of known determinants, by using~\eqref{eq.DJD}:
\vspace{-8pt}
\[
  D_{-1,2r-1}(2n+1)D_{0,2r}(2n-1) = D_{-1,2r-1}(2n)D_{0,2r}(2n)
  -\cancelto{0}{D_{0,2r-1}(2n)}D_{-1,2r}(2n),
\]
where $D_{0,2r-1}(2n)=0$ only if $n\geq r$. Therefore
\[
  \frac{D_{-1,2r-1}(2n+1)}{D_{-1,2r-1}(2n)} = \frac{D_{0,2r}(2n)}{D_{0,2r}(2n-1)}
  \qquad (n\geq r).
\]
\end{proof}

\begin{conj}\label{conj.reci}
There is a combinatorial reciprocity between determinants $D_{s,t}(n)$ which
just count cyclically symmetric rhombus tilings (the case when $s-t$ is even)
and determinants $D_{s,t}(n)$ which perform a weighted count (the case when
$s-t$ is odd). For example, we conjecture that
\[
  D_{2r-1,0}(2n+1) = D_{0,0}(2n-2r+2)\big|_{\mu\to 1-\mu-6n}
\]
for $n\geq r\geq1$. Note that, when setting $r,n,\mu$ to concrete integers, at
least one of the two determinants does not allow the combinatorial
interpretation given in Section~\ref{sec.comb}, for instance, because the hole
is larger than the hexagon.
\end{conj}

We would like to point out that special instances (setting the parameter~$r$
to a concrete integer) of the results presented in this last section, in
particular Conjectures~\ref{conj.famC} and~\ref{conj.famD}, may be provable in
the same manner as the results in Section~\ref{sec.lemmas}.  However, we don't
see how to use this computer algebra approach to prove them for symbolic~$r$,
since the extra parameter appears also in the Kronecker delta.

These conjectures are along the same line as Conjecture~37 in \cite[page
  50]{Krattenthaler05}, which, for the same reason, we have not been able to
prove in~\cite{KoutschanThanatipanonda13}.  In this related family of
determinants, the Kronecker delta is multiplied by~$-1$. Obviously, they count
the same kind of objects, but total count vs. weighted count change their
roles.  It would be worthwhile to investigate the connections between these
determinants and our determinant~$D_{s,t}(n)$, in the spirit of
Conjecture~\ref{conj.reci}.  First experiments suggest that also the
determinants $\tilde{D}_{s,t}(n)$ with negative Kronecker delta comprise
several infinite families that have nice evaluations or quotients. The
analysis of those should not be too different from what we did in the present
paper. Another interesting direction of research would be to find $q$-analogs
of all these determinants.

\paragraph{Acknowledgment.}
We are grateful to Christian Krattenthaler for helpful comments on an earlier
draft of this paper, and for explaining some of the combinatorial background
to us, and to Elaine Wong for proofreading and useful suggestions.  We thank
the anonymous referees for their reports that helped us improve the paper
significantly.

\bibliographystyle{plain}

\end{document}